\documentclass[11pt, reqno]{amsart}

\usepackage{amssymb,amscd,amsthm, verbatim,amsmath,color,fancyhdr, mathrsfs}
\usepackage{graphicx}
\usepackage{turnstile}
\usepackage{slashed}
\usepackage{tikz-cd}
\usepackage{enumitem}
\usepackage[hmargin=1.4in, vmargin=1.4in]{geometry}

\theoremstyle{definition} 
\newtheorem{thm}{Theorem}[section]
\newtheorem{prop}[thm]{Proposition}
\newtheorem{lem}[thm]{Lemma}
\newtheorem{rem}[thm]{Remark}
\newtheorem{cor}[thm]{Corollary}

\newtheorem{ex}[thm]{Example}
\newtheorem{dfn}[thm]{Definition}
\usepackage{xcolor}
\definecolor{softblue}{RGB}{60,120,200}
\definecolor{softred}{RGB}{170,60,60}

\usepackage[
  colorlinks=true,
  linkcolor=softred,   
  citecolor=softblue,  
  urlcolor=softblue
]{hyperref}
\usepackage{amsthm}

\newtheorem{theorem}{Theorem}
\newenvironment{intro}
  {%
    \setcounter{theorem}{0}%
  }
  {%
  }
\title[Linking at Infinity and Scalar Curvature Decay]{Linking at Infinity and Scalar Curvature Decay on Non-Compact Manifolds}

\numberwithin{equation}{section}

\author{Shunichiro Orikasa}
\address{Department of Mathematics, Graduate School of Science, Kyoto University, Sakyoku, Kyoto 606–8502, Japan}
\email{orikasa.shunichiro.34x@st.kyoto-u.ac.jp}
\date{\today}

\begin{document}
\begin{abstract}
We study complete non-compact manifolds of positive scalar curvature,
with a focus on how curvature decay is constrained by topology at infinity.
Our first main result shows that topological linking at infinity
forces polynomial decay of scalar curvature on manifolds
of weakly bounded geometry.
This result provides a conceptual generalization of recently discovered
examples of metrics with quadratic scalar curvature decay.

Building on this decay mechanism, we develop an obstruction theory
localized at the ends of non-compact manifolds.
Using $\mu$--bubble exhaustions together with the analysis of stable
minimal hypersurfaces and index theory,
we obtain qualitative obstructions to uniformly positive scalar curvature
on individual ends.

\end{abstract}
\maketitle

\begin{intro}

\section{Introduction}

The study on compact manifolds admitting metrics of positive scalar curvature
has seen substantial progress over the past decades.
Beginning with the resolution of the Geroch conjecture through the seminal works
of Schoen--Yau \cite{schoen1979existence} and Gromov--Lawson \cite{gromov1983positive},
a rich theory has been developed.
Moreover, even in the noncompact setting, obstruction theories for complete
metrics with \emph{uniformly} positive scalar curvature are relatively well understood,
notably via minimal hypersurface techniques \cite{schoen1979existence}
and Dirac operator methods \cite{gromov1983positive}.

In contrast, once the assumption of uniform positivity is removed,
the situation changes dramatically.
When the scalar curvature is allowed to decay at infinity,
the relationship between the topology at infinity and the behavior of scalar curvature
becomes highly delicate.
This is already reflected in the recent progress on quadratic decay estimates
for positive scalar curvature metrics using band width arguments
\cite{gromov2018metric}.
Despite these advances, a systematic understanding of how the decay of positive scalar curvature
interacts with the topology at infinity of noncompact manifolds
remains largely out of reach.

This difficulty is vividly illustrated by the following fundamental problem.
Once one allows for rapidly decaying positive scalar curvature,
very little is known about the classification of manifolds admitting such metrics.
For instance, the following conjecture remains open:
if $(M^3,g)$ is a contractible $3$--manifold equipped with a complete Riemannian metric
of positive scalar curvature $R>0$, must $M$ be diffeomorphic to $\mathbb{R}^3$?
In the case of positive Ricci curvature, the result was established by Schoen and Yau \cite{SchoenYau+2016+209+228}, who also provided a criterion to obstruct PSC on non-compact manifolds (see Section 3 thereof). 
Furthermore, the conjecture is known to hold under the stronger assumption $R \ge \sigma > 0$ for a constant $\sigma > 0$ \cite{chang2010taming}, which highlights the subtle nature of the $R>0$ case. 
We also note very recent progress regarding the topology of such manifolds in \cite{chodosh20253manifoldspositivescalarcurvature}.

The central question underlying this work is formulated as follows.

\begin{quote}
\emph{What is the relationship between the decay of positive scalar curvature
and the topology at infinity of noncompact manifolds?}
\end{quote}

In this paper, we provide a partial advance toward this question.
More precisely, we develop new tools to detect quantitative obstructions
arising from the topology at infinity, both for the nonexistence of uniformly positive
scalar curvature metrics and for geometric estimates in the presence of such metrics.

Our approach is based on techniques originating in geometric measure theory
that have been developed over recent years
\cite{Gromov2023FourLectures,chodosh2024generalized,chodosh2024complete}.
Within this framework, we study index theory
on (generally noncompact) minimal hypersurfaces
that arise from topological twisting at infinity,
thereby extending classical Dirac operator methods
in the spirit of \cite{gromov1983positive}
to a noncompact and asymptotic setting.

\subsection{Background and motivation}

A unifying perspective of this paper is the idea of \emph{localization}, which has
recently played a central role in scalar curvature geometry
\cite{Gromov2023FourLectures,chodosh2024generalized}.
A second guiding principle is that of \emph{spectral positivity}, whereby
spectral conditions on curvature lead to geometric and topological consequences
analogous to those arising from pointwise positivity of scalar curvature.

From a coarse geometric viewpoint, such localization allows one to assemble
large--scale information from geometric data at increasing scales
\cite{chodosh2023classifying}.
However, the $\mu$--bubble exhaustion property relies crucially on strict
positivity of scalar curvature and appears to be less well--suited to the study of metrics whose
scalar curvature decays rapidly at infinity.

This observation motivates a return to a careful analysis of minimal
hypersurfaces in the study of decaying positive scalar curvature.

\subsection{Topology at infinity and decay of scalar curvature}

Two--sided stable minimal hypersurfaces play a role in higher--dimensional
geometry analogous to that of stable geodesics in dimension one.
Given a minimal immersion
$
M^{n-1}\to (X^n,g)
$
with trivial normal bundle, stability is characterized by the inequality
\[
\int_M \bigl(|A_M|^2+\operatorname{Ric}_g(\nu,\nu)\bigr)\phi^2
\;\le\;
\int_M |\nabla \phi|^2,
\]
which holds for all compactly supported test functions~$\phi$.
This variational condition allows the hypersurface to ``sense" curvature of the ambient manifold through the normal Ricci term

When the ambient scalar curvature satisfies a lower bound $R_g\ge 1$,
the induced metric on~$M$ admits positive scalar curvature
\cite{schoen1979existence}.
This observation has become a foundational tool in the study of manifolds with
$R_g>0$.
The situation becomes more delicate once one allows the hypersurface to be
complete and non--compact.
In dimension three, however, the theory remains remarkably rigid.
Indeed, for a complete two--sided stable minimal surface
$
M^2\to (X^3,g),
$
nonnegativity of the scalar curvature forces the induced metric to be conformal to
either the plane or the cylinder~\cite{fischer1980structure}.
Even more strikingly, a positive scalar curvature lower bound $R_g\ge 1$
excludes the non--compact case altogether, forcing~$M$ to be compact
\cite{gromov1983positive}.

Recent work by Chodosh and Li \cite{chodosh2024stable} shows that a strong classification result persists in the flat ambient setting: any complete two-sided stable minimal hypersurface $M^3 \to \mathbb{R}^4$ must be flat. 
Furthermore, this classification has recently been extended to $\mathbb{R}^5$ \cite{chodosh2025stable} and $\mathbb{R}^6$ \cite{mazet2024stableminimalhypersurfacesmathbb}, whereas the phenomenon is known to fail for $\mathbb{R}^n$ with $n > 7$, with the case $n=7$ remaining a major open problem.
It is natural to
ask whether similar conclusions hold in curved ambient manifolds.

As pointed out in \cite{chodosh2024complete}, the answer turns out to be negative in dimension four. While in three dimensions scalar curvature bounds force compactness and Ricci positivity forbids existence, neither phenomenon survives in higher dimension. There exist non--compact stable minimal hypersurfaces in $4$--manifolds with uniformly positive scalar curvature (cf. {\cite[Example~1]{chodosh2024complete}}), as well as stable hypersurfaces in metrics with strictly positive sectional curvature. In particular, there exist rotationally symmetric metrics on $\mathbb{R}^4$ with strictly positive sectional curvature that contain totally geodesic copies of $\mathbb{R}^3$ (cf. {\cite[Example~2]{chodosh2024complete}}).

\begin{ex}[{\cite[Example~1]{chodosh2024complete}}]
\label{ex:psc_noncompact_stable}
Let $(X^2,g)$ be a closed oriented surface admitting a complete unit--speed stable
geodesic $\sigma:\mathbb{R}\to X^2$.
For $\varepsilon>0$ sufficiently small, the product manifold
$
(X^2,g)\times S^2(\varepsilon)
$
has scalar curvature bounded below by a positive constant.
Nevertheless, the product immersion
$
\sigma\times S^2(\varepsilon)
$
defines a complete, unbounded, two--sided stable minimal hypersurface.
\end{ex}
Our perspective is motivated by classical curvature estimates for entire high-dimensional minimal graphs. Stability imposes a quantitative decay estimate on the second fundamental form \cite[Theorem~11.4]{gromov1983positive}, forcing at most quadratic decay at infinity—a phenomenon already appearing in the scalar curvature setting \cite[Corollary~10.15]{gromov1983positive} and, more recently, in \cite{gromov2018metric}.
We also note that scalar curvature decay has been recently established on uniformly contractible manifolds via operator-theoretic methods \cite{wang2024decay}.

\subsection{Main Theorems}

The first and central theme concerns decay phenomena for positive scalar curvature
detected at infinity.
Our main result in this direction, Theorem~\ref{thm:link_infinity},
exhibits a new mechanism forcing scalar curvature decay,
arising from topological linking at infinity.
This theorem already yields
concrete polynomial decay estimates; see Corollary~\ref{cor:thm:link_infinity}.

The second theme addresses local obstructions to uniformly positive scalar curvature
in the non-compact setting.
We introduce higher-dimensional analogues of the classical small circle obstruction
of Gromov--Lawson~\cite{gromov1983positive},
formulated in terms of embedded \emph{small tori}.
These results show that uniform positivity may already fail on a single end;
see Theorem~\ref{thm:small_torus_obstruction}.

The third theme develops obstructions arising from topology at infinity, encoded by cohomology classes detected by a single proper ray. We prove that uniformly positive scalar curvature is obstructed if such a class decomposes as a cup product of degree-one classes (Theorem~\ref{thm:Ray_obstruction}). Furthermore, as a supplementary analysis, Appendix~\ref{sec:appendix_quantitative} provides a quantitative refinement of this result, establishing universal lower bounds on the $1$--dilation of proper maps (Theorem~\ref{thm:quantitative_sys_at_infinity}).

\medskip
\paragraph{\textbf{Linking at infinity and scalar curvature decay.}}\mbox{}
Our first main result, Theorem~\ref{thm:link_infinity},
reveals a new decay phenomenon for positive scalar curvature detected through
stable minimal hypersurfaces.
The precise notion of scalar curvature decay will be introduced and systematically
studied in Section~\ref{sec:sc_decay}.
The terminology appearing in Theorem~\ref{thm:link_infinity} will also be explained there.
\begin{theorem}[Theorem~\ref{thm:link_infinity}]
There exists a constant $C > 0$ such that the following statement holds.
Let $(X^4,g)$ be a complete oriented Riemannian manifold with weakly bounded geometry
and non-negative scalar curvature.
Let $\pi \colon X \to \mathbb{R}^3$ be a proper smooth map,
and let $F = \pi^{-1}(y)$ be the inverse image of a regular value.
Assume that $H_3(X)=0$ and that there exists
\begin{equation}
0 \neq \alpha \in \ker\bigl(H_2^\infty(X) \to H_2(X)\bigr)
\quad \text{with} \quad
\operatorname{Lk}(F,\alpha) \neq 0 .
\label{eq:alpha-linking-assumption}
\end{equation}
For $\rho>0$, set
$
\pi_\rho := \pi\big|_{\pi^{-1}(B^3(\rho))} \colon \pi^{-1}(B^3(\rho)) \to B^3(\rho).$
Then for all $\rho > 0$,
\begin{equation}
\min_{\pi^{-1}(B^3(\rho))} R_g
\;\le\;
C
\left(\frac{\operatorname{Dil}_1(\pi_\rho)}{\rho}\right)^2 .
\end{equation}
\end{theorem}
\begin{rem}
Let us briefly clarify the terminology in the theorem. 
The space $H_2^\infty(X) := \varprojlim_K H_2(X \setminus K)$ denotes the homology at infinity (see Section \ref{sec:prelim}). 
For a class $\alpha \in \ker(H_2^\infty(X) \to H_2(X))$, its representative $\alpha_K$ bounds a $3$--chain $W_K$ in $X$ for sufficiently large $K$; the linking number at infinity $\operatorname{Lk}(F,\alpha)$ is defined as the stabilized intersection number of the fiber $F$ with $W_K$ (see Section \ref{sec:sc_decay}). 
Finally, $\operatorname{Dil}_1(\pi_\rho)$ is the $1$--dilation (i.e., the supremum of the operator norm of the differential) of the restricted map $\pi_\rho$.
\end{rem}
Conceptually, Theorem~\ref{thm:link_infinity} shows that
topological complexity at infinity,
encoded by a nontrivial linking condition in homology at infinity,
forces the scalar curvature to decay at a controlled rate.
This phenomenon is detected via the formation of stable minimal hypersurfaces
carrying natural spin structures,
which allow us to invoke index-theoretic arguments in a non-compact setting.
This approach is conceptually analogous to the Bochner method on stable minimal hypersurfaces (cf.~\cite{miyaoka1993l2, bei20262}). In our setting, we instead investigate harmonic solutions of the spinor Dirac operator to capture the curvature decay.
The resulting obstruction prevents scalar curvature from remaining uniformly large
at infinity.

The structural similarity between this mechanism and the stable minimal hypersurfaces
appearing in positively curved $4$--manifolds,
as described in \cite[Example~2]{chodosh2024complete},
is particularly noteworthy.

As an application of Theorem~\ref{thm:link_infinity}, we investigate a complete smooth
Riemannian metric $g$ on $\mathbb{R}^2 \times T^2$ with positive scalar curvature and
quadratic decay at infinity, recently discovered in \cite{gromov2018metric}.
After lifting this metric to an infinite cyclic cover
$X = \widetilde{\mathbb{R}^2 \times T^2}$,
we recover a polynomial decay estimate for the scalar curvature.
Although the resulting bound is not optimal compared with the explicit curvature
computation, it relies on a new technical mechanism of a different nature.

More precisely, for the natural proper map
\[
\pi \colon X \to \mathbb{R}^3 ,
\]
we obtain
\[
\min_{\pi^{-1}(B^3(\rho))} R_g
\sim
\frac{1}{\rho^2}
\;\lesssim\;
\left(\frac{\operatorname{Dil}_1(\pi_\rho)}{\rho}\right)^2
\sim
\rho^{-2\alpha}.
\]
Here the explicit quadratic decay follows from a direct computation,
while the upper bound is a consequence of Theorem~\ref{thm:link_infinity},
illustrating how linking at infinity enforces scalar curvature decay
at large scales.
This will be discussed in Section~\ref{sec:sc_decay}.

Corollary~\ref{cor:thm:link_infinity} illustrates how Theorem~\ref{thm:link_infinity}
naturally leads to polynomial decay estimates for scalar curvature under
quantitative domination assumptions.
In particular, it provides a flexible framework that applies to
metrics exhibiting slow decay,
including the example constructed by Gromov~\cite{gromov2018metric}.

\begin{cor}[Corollary~\ref{cor:thm:link_infinity}]
Let $(X^4,g)$ be a complete oriented Riemannian manifold
with weakly bounded geometry and non-negative scalar curvature.
If $(X,g)$ is $\alpha$--weakly dominated over $\mathbb R^3$ (see Definition~\ref{def:weak_domination}) by a map
\(
\pi \colon X \to \mathbb R^3,
\)
then for all $\rho>0$,

\[
\min_{\pi^{-1}(B^3(\rho))} R_g
\;\lesssim\;
\rho^{-2\alpha} .
\]
\end{cor}

\medskip
\paragraph{\textbf{Higher-dimensional small-circle obstructions}}\mbox{}
Our second main result develops a localized obstruction to uniformly positive scalar curvature.
The key geometric object is a \emph{small torus},
which can be regarded as  a natural higher-dimensional analogue of the
\emph{small circle} obstruction introduced by Gromov and Lawson
in their study of open $3$--manifolds; see, for example,
\cite{gromov1983positive}.
We defer the precise definition and a systematic treatment of this notion
to Section~\ref{sec:small_torus}.

Roughly speaking, a small torus is an embedded torus whose normal geometry,
together with its interaction with the topology at infinity of the ambient
manifold, allows it to be detected by stable minimal hypersurfaces arising
from $\mu$--bubble exhaustions.

Our second main result shows that the presence of such a torus already
precludes the existence of complete metrics of uniformly positive scalar curvature.

\begin{theorem}[Theorem~\ref{thm:small_torus_obstruction}]
For $n \leq 7$, let $X$ be an open $n$--manifold with $H_1(X)$ finitely generated.
If $X$ contains a small torus, then $X$ admits no complete metric of
uniformly positive scalar curvature.
Moreover, there exists an end $X_+ \subset X$ such that there is no
complete Riemannian metric $g$ on $X$ satisfying
\[
R_g \ge \sigma > 0 \quad \text{on } X_+ .
\]
\end{theorem}
Unlike settings where the obstruction is detected via a compactification of the ends,
as in \cite[Theorem~3]{chodosh2024generalized},
our approach applies directly to non-compact geometries.
The obstruction is intrinsically local at infinity:
the presence of a single suitably embedded torus on an end already precludes
uniform positivity of scalar curvature on that end.
\medskip
\paragraph{\textbf{Obstructions from cohomology at infinity.}}\mbox{}
The third theme further develops obstructions arising from topology at infinity,
now encoded by cohomology classes detected by a single proper ray.
The proper ray $\gamma$ canonically determines a cohomology class at infinity,
given by its Poincaré dual $\alpha_\gamma$, as defined in
Proposition~\ref{prop:ray-infinity}.

In Section~\ref{sec:psc_cycle}, we prove that
if such a class decomposes as a cup product of degree--one classes
coming from infinity,
then uniformly positive scalar curvature is impossible in dimensions up to seven.

\begin{theorem}[Theorem~\ref{thm:Ray_obstruction}]
Let $n \le 7$.
Let $X^n$ be a connected oriented non-compact manifold, and let
$\varepsilon$ be an isolated end of $X$.
Assume that there exists a proper ray
$\gamma \subset \varepsilon$
whose Poincar\'e dual at infinity
\[
\alpha_\gamma \in H^{n-1}_\infty(\varepsilon)
\]
admits a decomposition
\[
\alpha_\gamma
=
u^1 \cup u^2 \cup \cdots \cup u^{n-1},
\qquad
u^i \in H^1_\infty(\varepsilon).
\]
Then there exists no complete Riemannian metric $g$ on $X$
satisfying
\[
R_g \ge \sigma > 0
\quad \text{on } \varepsilon .
\]

\end{theorem}

A key technical tool in the proof is the $\mu$--bubble exhaustion
\cite[Section~3.7.2]{Gromov2023FourLectures},
\cite[Proposition~3.1]{otis2024complete},
which produces a family of hypersurfaces with spectrally positive scalar curvature
adapted to the large--scale geometry.

\subsection{Organization of the paper}

The paper is organized as follows.
In Section~\ref{sec:mu-bubble-exhaustions}, we review background material on the method of
$\mu$--bubble exhaustions, including the existence of $\mu$--bubbles and
the resulting exhaustion theory
(cf.~Lemma~\ref{lem:existence_mu_bubbles} and
Theorem~\ref{thm:exhaustion}).
Section~\ref{sec:prelim} introduces the basic topological framework used
throughout the paper, including the notions of homology and cohomology
at infinity and their intersection properties
(see Definitions~\ref{def:Homologyatinfinity} and
\ref{def:Cohomologyatinfinity}, and
Lemma~\ref{lem:intersection-hypersurface}).
In Section~\ref{sec:sc_decay}, we prove the main decay estimate based on
linking at infinity
(Theorem~\ref{thm:link_infinity}), together with a model example by Gromov's constructions
(Example~\ref{ex:gromov_quadratic_decay}).
Section~\ref{sec:small_torus} establishes new obstructions to the existence of complete metrics of
positive scalar curvature
(Theorem~\ref{thm:small_torus_obstruction}).
In Section~\ref{sec:psc_cycle}, we study positive scalar curvature in relation to cohomology classes at infinity, establishing qualitative obstructions (Theorem~\ref{thm:Ray_obstruction}). 
Appendix~\ref{sec:weakly_bdd} reviews the notion of weakly bounded geometry
and collects auxiliary curvature estimates
(see Lemma~\ref{lem:stable-curvature-estimate}),
while Appendix~\ref{sec:construction_qdecay} describes a model metric (cf.~\eqref{eq:metric}) with
quadratic scalar curvature decay and the associated dilation estimates.
Finally, Appendix~\ref{sec:appendix_quantitative} develops a quantitative refinement of obstructions from Theorem~\ref{thm:Ray_obstruction}, yielding lower bounds on the $1$--dilation of proper maps (Theorem~\ref{thm:quantitative_sys_at_infinity}).
\section*{Conventions and Notation}

This section fixes analytic and geometric conventions used throughout the paper.
All sign choices are made so that curvature positivity corresponds to spectral
positivity of the associated differential operators.
\subsection*{Riemannian conventions}

All manifolds are assumed to be smooth, connected, and oriented unless stated otherwise.
For a smooth function $f$, we denote by $\nabla f$ its gradient and by
$\operatorname{div}_g$ the Riemannian divergence.
The Laplace--Beltrami operator acting on functions is defined by
\[
\Delta := \operatorname{div}_g \circ \nabla .
\]
With this convention, the operator $-\Delta$ is non-negative.

Let $(\Sigma^m,h)$ be an $m$--dimensional Riemannian manifold with $m \ge 3$.
The conformal Laplacian is defined by
\[
L_\Sigma
:=
- \frac{4(m-1)}{m-2}\,\Delta_\Sigma + R_\Sigma .
\]
We write $R_g$ and $\operatorname{Ric}_g$ for the scalar curvature and the Ricci curvature
of a Riemannian manifold $(M,g)$, respectively.

\subsection*{Caccioppoli sets and reduced boundary}

Following \cite{simon2014introduction}, a measurable set $\Omega \subset (M^n,g)$ is called a \emph{Caccioppoli set}
if its characteristic function $\chi_\Omega$ has bounded variation.
Equivalently, the distributional gradient $D\chi_\Omega$ is a finite
Radon measure.

The \emph{reduced boundary} of $\Omega$ is denoted by $\partial^*\Omega$.
It consists of those points $x \in M$ for which there exists a measure--theoretic
unit normal vector $\nu_\Omega(x)$ satisfying
\[
\lim_{r \to 0}
\frac{D\chi_\Omega(B_r(x))}{|D\chi_\Omega|(B_r(x))}
= \nu_\Omega(x),
\quad
|\nu_\Omega(x)| = 1.
\]

The perimeter of $\Omega$ in an open set $U \subset M$ is given by
\[
\mathrm{Per}(\Omega,U)
:=
|D\chi_\Omega|(U)
=
\mathcal{H}^{n-1}(\partial^*\Omega \cap U).
\]

All Hausdorff measures $\mathcal{H}^k$ are taken with respect to the
distance induced by the ambient Riemannian metric $g$.

\subsection*{Dirac operators}

All spin manifolds are assumed to be equipped with a fixed spin structure.
If $E \to M$ is a Hermitian vector bundle with unitary connection $\nabla^E$,
we denote by $\slashed{D}_E$ the Dirac operator twisted by $E$, $
\slashed{D}_E \colon \Gamma(\mathbb{S}\otimes E) \to \Gamma(\mathbb{S}\otimes E).
$
The corresponding Weitzenb\"ock formula is
\[
\slashed{D}_E^2
=
\nabla^* \nabla
+
\frac{1}{4} R_g
+
\mathcal{R}_E ,
\]
where $\mathcal{R}_E$ denotes the curvature term induced by $\nabla^E$.
Positivity or negativity of $\mathcal{R}_E$ is always understood in the sense of
pointwise Hermitian endomorphisms.

\subsection*{Dilation}

Let $(X,g_X)$ and $(Y,g_Y)$ be Riemannian manifolds and
$f \colon X \to Y$ a smooth map.
The \emph{$1$--dilation} of $f$ is defined by
\[
\operatorname{Dil}_1(f)
:=
\sup_{x \in X} \| df_x \|_{\mathrm{op}}
=
\sup_{x \in X}
\sup_{v \in T_x X \setminus \{0\}}
\frac{|df_x(v)|_{g_Y}}{|v|_{g_X}} .
\]
In this paper, all dilation estimates refer to $\operatorname{Dil}_1$ unless explicitly stated
otherwise.

\section{\texorpdfstring{$\mu$}{mu}-Bubble Exhaustions}\label{sec:mu-bubble-exhaustions}
The basic strategy we use originates in the work of Schoen and Yau
\cite{schoen1979existence,schoen1987structure}, who analyzed the second variation of area for
area-minimizing hypersurfaces. Their key observation is that if
$\Sigma^{n-1}$, $n \ge 3$, is a two-sided stable minimal hypersurface in a
Riemannian manifold $(M^n,g)$ with positive scalar curvature, then one can
conformally modify the induced metric $g|_\Sigma$ so that $\Sigma$ itself
admits a metric of positive scalar curvature. When combined with suitable
existence results for minimal hypersurfaces, this leads to strong
topological obstructions to positive scalar curvature.

From a structural point of view, the Schoen--Yau argument depends much more
on the second variation inequality than on minimality itself. Motivated by
this observation, and following \cite{chodosh2024generalized}, we adopt an idea of Gromov
\cite{Gromov2023FourLectures} in which minimal hypersurfaces are replaced by critical points
of a modified area functional. The advantage of this perspective is that
the associated minimization problem can be localized, while the resulting
second variation formula continues to impose obstructions to positive
scalar curvature in a range of geometric situations.

\subsection{Review of \texorpdfstring{$\mu$}{mu}-bubbles and band width}

We first focus on the existence of $\mu$-bubbles. For this purpose, we introduce the
notion of a Riemannian band. This section is based on \cite{richard2023small}, \cite{chodosh2024generalized}.
For the sake of completeness, we include a proof of the known properties of $\mu$-bubbles.
\begin{dfn}[Riemannian band]
A triple $(M,\partial^\pm,g)$ is called a \emph{Riemannian band} if $(M,g)$ is a compact
Riemannian manifold with nonempty boundary, and $\partial M$ decomposes as
\[
\partial M = \partial^- \sqcup \partial^+,
\]
where $\partial^-$ and $\partial^+$ are disjoint unions of boundary components.
\end{dfn}

Let $(M^n,g,\partial^\pm)$ be a Riemannian band and let $\Sigma_0$ be a fixed closed
two-sided hypersurface which separates $\partial^-$ and $\partial^+$. Denote by $\Omega_0$ the region
bounded by $\partial^-$ and $\Sigma_0$. We consider the class
\[
\mathcal{C} := \{\Omega \subset M \text{ Caccioppoli set } : \Omega \Delta \Omega_0
\Subset \mathring{M}\}.
\]

Given a smooth function $h$ defined on the interior $\mathring{M}$, we introduce the
functional
\[
\mathcal{A}_h(\Omega)
:= \mathcal{H}^{n-1}(\partial^*\Omega \cap \mathring{M})
 - \int_{\mathring{M}} (\chi_\Omega - \chi_{\Omega_0})\, h \, d\mathcal{H}^n_g .
\]

$\Omega$ which minimizes $\mathcal{A}_h$ in the class $\mathcal{C}$ is called a \emph{$\mu$-bubble}. One of the important advantage of $\mu$-bubbles is that they always exist  a given Riemannian band for a suitably imposed function $h$. We begin with the following lemma.

We next recall the first and second variation formulas for the $\mu$-bubble
functional. Let $\{\Omega_t\}$ be a smooth one-parameter family of regions with
$\Omega_0=\Omega$, and let $\phi$ denote the normal speed at $t=0$. Then the first
variation is given by
\[
\frac{d}{dt}\mathcal{A}_h (\Omega_t)
=
\int_{\partial \Omega_t} (H - h)\, \phi,
\]
where $H$ denotes the mean curvature of $\partial \Omega_t$. In particular, a
$\mu$-bubble satisfies
\[
H = h \quad \text{along } \Sigma := \partial \Omega .
\]

Assuming that $\partial \Omega = \Sigma$ satisfies $H = h$, the second variation of
$\mathcal{A}_h$ is given by
\begin{equation}\label{eq:mu-bubble-second-variation}
\left.\frac{d^2}{dt^2}\right|_{t=0} \mathcal{A}_h (\Omega_t)
=
\int_\Sigma
\left(
|\nabla_\Sigma \phi|^2
-
\frac12
\bigl(
R_M - R_\Sigma + |A|^2 + h^2 + 2\langle \nabla_M h, \nu \rangle
\bigr)
\phi^2
\right).
\end{equation}
Here $R_M$ denotes the scalar curvature of $M$, $R_\Sigma$ the scalar
curvature of the induced metric on $\Sigma$, $A$ the second fundamental form of
$\Sigma$, and $\nu$ the unit normal vector.

The minimizer $\Omega$ obtained above satisfies the stability inequality
\[
\left.\frac{d^2}{dt^2}\right|_{t=0} \mathcal{A}_h (\Omega_t) \ge 0
\quad
\text{for all } \phi \in C^\infty(\Sigma).
\]

The existence of a minimizer for $\mathcal{A}_h$ can be understood through the first variation formula. 
Indeed, for a suitable choice of the function $h$ (as specified in Lemma \ref{lem:existence_mu_bubbles} below), the boundary component $\partial^- M$ satisfies $H_{\partial^- M} - h = -\infty$. 
Consequently, $\partial^- M$ serves as a strict barrier for any minimizing sequence.
An analogous barrier arises from $\partial^+ M$. As a result,
any minimizing sequence $\{\Omega_j\}$ is forced to remain uniformly away from the
boundary, and in particular satisfies
\[
\Omega_j \triangle \Omega_0 \Subset \mathring{M}.
\]

Once this localization is established, standard compactness results from geometric
measure theory apply and produce a minimizer $\Omega$ with
$\Omega \triangle \Omega_0 \Subset \mathring{M}$. Moreover, since the functional
$\mathcal{A}_h$ differs from the area functional only by a lower-order perturbation
at small scales in the interior of $M$, classical regularity theory implies that
the boundary $\Sigma=\partial\Omega$ is a smooth compact hypersurface contained in
$\mathring{M}$. In dimensions $n\ge8$, singularities may occur and additional
arguments are required.
This heuristic discussion can be made precise as follows.
\begin{lem}
\label{lem:existence_mu_bubbles}
Assume $n \le 7$ and that the function $h$ satisfies
\begin{equation}\label{eq:mu-bubble-boundary-limits}
\lim_{x \to \partial^-} h(x) = +\infty,
\qquad
\lim_{x \to \partial^+} h(x) = -\infty.
\end{equation}
Then there exists a smooth minimizer $\widehat{\Omega} \in \mathcal{C}$ of
$\mathcal{A}_h$.
\end{lem}

\begin{proof}
We set
\begin{equation*}\label{eq:def-I}
I := \inf \{ A_h(\Omega) \mid \Omega \in \mathcal C \}.
\end{equation*}
We first show that $I > -\infty$.
For $s>0$, define
\begin{equation*}\label{eq:def-sigma-s}
\Sigma_s^\pm := \{ x \in M^\circ \mid \operatorname{dist}(x,\partial^\pm)=s \}.
\end{equation*}
For sufficiently small $s$, the hypersurfaces $\Sigma_s^\pm$ form smooth
foliations in a neighborhood of $\partial^\pm$.
By the assumption, we may assume that there exists a small constant $s_0>0$ such that
\begin{equation}\label{eq:mean-curvature-barriers}
H_s^- \le h
\quad \text{and} \quad
H_s^+ \le -\, h
\qquad \text{for all } s \le s_0,
\end{equation}
where $H_s^\pm$ denotes the mean curvature of $\Sigma_s^\pm$
with respect to the $\partial_s$--direction.
Let $\Omega_s^\pm$ be the region bounded by $\Sigma_s^\pm$ and $\partial^\pm$.
Possibly shrinking $s_0$, we can construct a smooth vector field $X$ such that
$X=\partial_s$ on $\Omega_{s_0}^\pm$.
Then it is clear that
\begin{equation}\label{eq:div-X-minus}
\operatorname{div}_g X = H_s^- \le h \quad \text{in } \Omega_{s_0}^-,
\end{equation}
and
\begin{equation}\label{eq:div-X-plus}
\operatorname{div}_g X = H_s^+ \le -h \quad \text{in } \Omega_{s_0}^+.
\end{equation}

The generalized Gauss--Green theorem
(see \cite[Chapter~3]{simon2014introduction})
implies that
\begin{align*}
\int_{\Omega_{s_0}^-\setminus \Omega} \operatorname{div}_g X \, d\mathcal H^n_g
&\ge H^{n-1}\bigl(\partial \Omega_{s_0}^-\setminus \Omega\bigr)
 - H^{n-1}\bigl(\partial^*\Omega \cap \Omega_{s_0}^-\bigr),
\\
\int_{\Omega \cap \Omega_{s_0}^+} \operatorname{div}_g X \, d\mathcal H^n_g
&\ge H^{n-1}\bigl(\partial \Omega_{s_0}^+\cap \Omega\bigr)
 - H^{n-1}\bigl(\partial^*\Omega \cap \Omega_{s_0}^+\bigr).
\end{align*}
Therefore
\begin{equation}\label{eq:Ah-lower-bound}
A_h(\Omega)
\ge A_h\bigl(\Omega \cup \Omega_{s_0}^-\setminus \Omega_{s_0}^+\bigr)
\ge - C\, \mathcal H^n(M,g),
\qquad \forall\, \Omega \in \mathcal C,
\end{equation}
where $C$ is a universal constant such that
$|h|\leq C$ on $M\setminus(\Omega_{s_0}^+\cup \Omega_{s_0}^-)$.
Consequently,
\begin{equation}\label{eq:I-finite}
I > -\infty.
\end{equation}

We now establish the existence of a minimizer of $A_h$ in $\mathcal C$.
Let $\{\Omega_k\} \subset \mathcal C$ be a minimizing sequence such that
$A_h(\Omega_k) \to I$.
Using the calculation described above, we may assume that
\begin{equation}\label{eq:symmetric-difference-control}
\Omega_k \Delta \Omega_0 \subset M \setminus (\Omega^-_{s_0} \cup \Omega^+_{s_0}).
\end{equation}
For sufficiently large $k$, there holds
\begin{equation}\label{eq:perimeter-bound}
\mathcal H^{n-1}(\partial^* \Omega_k)
\le I + 1 + C\, \mathcal H^n(M,g).
\end{equation}
The compactness theorem for sets of finite perimeter yields, after extracting
a subsequence, a limit $\widehat{\Omega} \in \mathcal C$ with
\begin{equation}\label{eq:existence-minimizer}
A_h(\widehat{\Omega}) = I.
\end{equation}

It follows that $\widehat{\Omega}$ is a minimizer of $A_h$, and hence
$\partial \widehat{\Omega}$ is smooth by standard regularity theory
\cite{tamanini1984regularity}.
\end{proof}

\subsection{Construction of \texorpdfstring{$\mu$}{mu}-Bubble exhaustions}
A fundamental result due to Gromov asserts that
a complete Riemannian manifold with uniformly positive scalar curvature
admits an exhaustion by compact domains whose boundaries carry
positive scalar curvature metrics
\cite[Section~3.7.2]{Gromov2023FourLectures}.
A detailed proof via $\mu$--bubble constructions
is presented in \cite[Proposition~3.1]{otis2024complete}.

\begin{prop}
\label{prop:existence_band}
Let $3\le n\le 7$ and let $(M^n,g)$ be a Riemannian manifold with non-empty boundary
satisfying
$
R_M \ge \Lambda > 0 .
$
Assume that the boundary decomposes as
$
\partial M = \partial^- M \sqcup \partial^+ M,
$
where both components are non-empty.

Then there exists a constant
$
D(\Lambda):=\frac{2\pi}{\sqrt{\Lambda}}
$
such that if
\[
D:=d(\partial^- M,\partial^+ M) > D(\Lambda),
\]
there exists a smoothly embedded, closed, two-sided hypersurface
\[
\Sigma^{n-1} \subset \operatorname{int}(M)
\]
which admits a Riemannian metric of positive scalar curvature.
Moreover, if $n\ge4$, the first eigenvalue of the conformal Laplacian
$L_\Sigma$ satisfies
\[
\lambda_1(L_\Sigma)\ge \Lambda-\frac{\pi^2}{D^2}.
\]
\end{prop}
Proposition~\ref{prop:existence_band} serves as the basic local geometric input for what
follows.
It asserts that whenever a region has two boundary components separated by a
sufficiently large distance, one can find a closed hypersurface with positive
scalar curvature sitting between them.
In dimensions $n\ge4$, this hypersurface enjoys a quantitative form of
positivity, expressed by a uniform lower bound on the first eigenvalue of its
conformal Laplacian.
Crucially, both the existence of the hypersurface and the spectral estimate
depend only on the scalar curvature lower bound of the ambient manifold and on
the separation scale between the boundary components.

In the next step, we explain how this local construction can be combined with
the large--scale topology at infinity to produce global obstructions.
Roughly speaking, the topology forces the appearance of regions with widely
separated boundary components, while Proposition~\ref{prop:existence_band}
guarantees the presence of positively curved hypersurfaces spanning such
regions.
The tension between these two effects will ultimately lead to the desired
contradiction.

The key point is that this construction is flexible and can be applied
iteratively. By viewing a non-compact manifold as built from a sequence of
wide annular regions, Proposition~\ref{prop:existence_band} can be used as a
building block to produce hypersurfaces of positive scalar curvature at
arbitrarily large scales.
\begin{thm}[{\cite[Section~3.7.2]{Gromov2023FourLectures}},
{\cite[Proposition~3.1]{otis2024complete}}]\label{thm:exhaustion}
Let $n\le7$ and let $(M^n,g)$ be a complete non-compact Riemannian manifold
with scalar curvature $R_M\ge \Lambda>0$. Then there exists an exhaustion
\[
\Omega_1\subset\Omega_2\subset\Omega_3\subset\cdots\subset M
\]
by compact domains such that each boundary $\partial\Omega_i$ is smooth and
has positive scalar curvature in the spectral sense.

In particular, when $n\ge4$, the exhaustion may be chosen so that, for arbitrary large $D>D(\Lambda)$, writing
$\Sigma_i:=\partial\Omega_i$, one has
\[
\lambda_1(L_{\Sigma_i})\ge \delta(D)>0
\]
for all $i$, where
\[
\delta(\Lambda):=\Lambda-\frac{4\pi^2}{D^2}.
\]
\end{thm}

\begin{proof}[Proof of Proposition~\ref{prop:existence_band}]
The argument is based on constructing a suitable auxiliary function that
separates the two boundary components. It suffices to find a function
\begin{equation*}
\rho : M \longrightarrow \Bigl[-\frac{\pi}{2},\frac{\pi}{2}\Bigr]
\end{equation*}
with the following properties:
\begin{enumerate}
\item $\rho^{-1}(-\frac{\pi}{2})=\partial^- M$ and
      $\rho^{-1}(\frac{\pi}{2})=\partial^+ M$,
\item $\rho$ is smooth on $\operatorname{int}(M)$,
\item $\operatorname{Dil}(\rho)\le \frac{\pi}{D}$.
\end{enumerate}
Such a function can be obtained by starting from the distance function to
$\partial^- M$, smoothing near the cut locus, and rescaling so that
$\frac{\pi}{2}$ is a regular value.

We now define a function $h$ on $\operatorname{int}(M)$ by
\begin{equation*}
h(x):=-\frac{2\pi}{D}\tan(\rho(x)).
\end{equation*}
A direct computation shows that
\begin{equation*}
|\nabla_M h|\le \frac{2\pi^2}{D^2}\sec^2(\rho),
\end{equation*}
and hence
\begin{equation}\label{eq:h-inequality}
h^2+2\langle\nabla_M h,\nu\rangle
\;\ge\;
h^2-2|\nabla_M h|
=
\frac{4\pi^2}{D^2}\bigl(\tan^2\rho-\sec^2\rho\bigr)
=
-\frac{4\pi^2}{D^2}.
\end{equation}

Let $\Sigma$ be a hypersurface produced by the $\mu$--bubble construction
associated with $h$. Applying the stability inequality
\eqref{eq:mu-bubble-second-variation}, using the scalar curvature bound
$R_M\ge\Lambda$, and discarding the non-negative term $|\mathring A|^2$, we obtain
\begin{equation}\label{eq:stability-preliminary}
\int_\Sigma
|\nabla_\Sigma\varphi|^2
+\frac12 R_\Sigma\varphi^2
-\frac12\left(\Lambda-\frac{4\pi^2}{D^2}\right)\varphi^2
\;\ge\;0
\end{equation}
for all $\varphi\in C_c^\infty(\Sigma)$.

Setting
\begin{equation}\label{eq:def-D-Lambda}
D(\Lambda)^2=\frac{4\pi^2}{\Lambda}
\end{equation}
and assuming $D>D(\Lambda)$, we define
\begin{equation}\label{eq:def-delta}
\delta:=\Lambda-\frac{4\pi^2}{D^2}>0.
\end{equation}
Then
\begin{equation}\label{eq:psc-stability}
\int_\Sigma
|\nabla_\Sigma\varphi|^2
+\frac12 R_\Sigma\varphi^2
\;\ge\;
\frac{\delta}{2}\int_\Sigma\varphi^2 .
\end{equation}

If $n=3$, then $\Sigma$ is two-dimensional and $R_\Sigma=2K_\Sigma$.
Applying \eqref{eq:psc-stability} with $\varphi\equiv1$ on a connected component
$\Sigma'$ yields
\begin{equation*}
2\pi\chi(\Sigma')=\int_{\Sigma'}K_\Sigma>0,
\end{equation*}
and hence $\Sigma'$ is diffeomorphic to $S^2$ (or $\mathbb{RP}^2$ in the
non-orientable case), which admits positive scalar curvature.

If $n\ge4$, set $m=\dim\Sigma=n-1\ge3$. Since
\begin{equation*}
\frac{2(m-2)}{m-3}\ge1,
\end{equation*}
inequality \eqref{eq:psc-stability} implies
\begin{equation}\label{eq:stability-high-dim}
\int_\Sigma
\frac{2(m-2)}{m-3}|\nabla_\Sigma\varphi|^2
+\frac12 R_\Sigma\varphi^2
\;\ge\;
\frac{\delta}{2}\int_\Sigma\varphi^2 .
\end{equation}
Equivalently, the conformal Laplacian
\begin{equation*}
L_\Sigma:=-\frac{4(m-2)}{m-3}\Delta_\Sigma+R_\Sigma
\end{equation*}
is a positive operator. Let $u>0$ be its first eigenfunction,
$L_\Sigma u=\lambda u$ with $\lambda\ge\delta>0$.
Then the conformally rescaled metric
\begin{equation*}
\tilde g_\Sigma:=u^{\frac{4}{m-3}}g_\Sigma
\end{equation*}
has positive scalar curvature.
\end{proof}

\begin{proof}[Proof of Theorem~\ref{thm:exhaustion}]
Let $M$ be a non-compact manifold. We construct an exhaustion of $M$ by compact
domains whose boundaries admit positive scalar curvature. The construction is
inductive and relies on Proposition~\ref{prop:existence_band}.

Let $D>D(\Lambda)$ be the constant from Proposition~\ref{prop:existence_band} and fix a sufficiently
small $\varepsilon>0$.

\medskip
\noindent\textbf{Step 1: Construction of the initial domain.}
Choose a point $p\in M$ and let $\{K_j\}_{j=1}^\infty$ be a compact exhaustion of
$M$. We select a smooth compact domain $M'_1\subset M$ such that
\begin{equation}\label{eq:step1-containment}
K_1\cup B_{D+\varepsilon}(p)\subset \operatorname{int}(M'_1),
\qquad
d(p,\partial M'_1)>D+\varepsilon.
\end{equation}
Define
\begin{equation*}
A_1:=M'_1\setminus \operatorname{int}(B_\varepsilon(p)).
\end{equation*}
Its boundary decomposes as
\begin{equation*}
\partial A_1=\partial^-A_1\sqcup\partial^+A_1,
\qquad
\partial^-A_1=\partial B_\varepsilon(p)\cong S^{n-1},
\quad
\partial^+A_1=\partial M'_1.
\end{equation*}

Applying Proposition~\ref{prop:existence_band} to $(A_1,g|_{A_1})$, we obtain a compact domain
$W_1\subset A_1$ whose boundary satisfies
\begin{equation*}
\partial W_1=\partial^-A_1\cup\Sigma_1,
\end{equation*}
where $\Sigma_1$ is a smooth hypersurface admitting a metric of positive scalar
curvature. We then define
\begin{equation*}
\Omega_1:=B_\varepsilon(p)\cup W_1.
\end{equation*}
By construction, $\partial\Omega_1=\Sigma_1$ has positive scalar curvature.

\medskip
\noindent\textbf{Step 2: Inductive construction.}
Assume that for some $k\ge1$ we have constructed compact domains
\begin{equation*}
\Omega_1\subset\Omega_2\subset\cdots\subset\Omega_k
\end{equation*}
such that each $\Omega_j$ is smooth and $\partial\Omega_j=\Sigma_j$ admits
positive scalar curvature.

Choose a smooth compact domain $M'_{k+1}\subset M$ containing $\Omega_k$ and
satisfying
\begin{equation}\label{eq:distance-Mk1}
d(\partial M'_{k+1},\Sigma_k)>D.
\end{equation}
Set
\begin{equation*}
A_{k+1}:=M'_{k+1}\setminus \operatorname{int}(\Omega_k),
\end{equation*}
so that
\begin{equation}\label{eq:boundary-Ak1}
\partial A_{k+1}=\Sigma_k\sqcup \partial M'_{k+1},
\end{equation}
with $\Sigma_k$ serving as the inner boundary.

Applying Proposition~\ref{prop:existence_band} to $A_{k+1}$, we obtain a compact domain
$W_{k+1}\subset A_{k+1}$ whose boundary decomposes as
\begin{equation}\label{eq:boundary-Wk1}
\partial W_{k+1}=\Sigma_k\cup\Sigma_{k+1},
\end{equation}
where $\Sigma_{k+1}$ lies strictly in the interior of $A_{k+1}$ and each of its
connected components admits positive scalar curvature. We then define
\begin{equation}\label{eq:def-Omegak1}
\Omega_{k+1}:=\Omega_k\cup W_{k+1}.
\end{equation}
By construction, $\partial\Omega_{k+1}=\Sigma_{k+1}$ has positive scalar curvature.

\medskip
\noindent\textbf{Step 3: Exhaustion.}
Iterating this procedure yields a nested sequence of compact domains
\begin{equation*}
\Omega_1\subset\Omega_2\subset\Omega_3\subset\cdots\subset M.
\end{equation*}
By choosing $M'_i$ at each step to contain a sufficiently large sublevel set of a
proper exhaustion function on $M$, we ensure that
\begin{equation*}
M=\bigcup_{i=1}^\infty \Omega_i.
\end{equation*}
This completes the construction of an exhaustion of $M$ by compact domains with
boundary admitting positive scalar curvature.
\end{proof}

\end{intro}

\section{Topological Preliminaries}\label{sec:prelim}
Non-compact manifolds often carry rich and subtle topology hidden near
infinity.
In this section, we introduce several tools designed to extract and
organize this information.
The basic idea is that if non-trivial topology persists arbitrarily far
out, then it must eventually manifest itself through the appearance of
hypersurfaces escaping to infinity.

Later, these hypersurfaces will be realized geometrically using minimal
surface techniques and $\mu$--bubbles, thereby providing a link between
topology at infinity and obstructions coming from scalar curvature.

\subsection{Ends of Non-compact Manifolds}
Throughout this section, all spaces are assumed to be realized by connected locally finite
simplicial complexes.
Moreover, all groups are assumed to be finitely generated unless otherwise specified. The content of this section is based on \cite{balacheff2025complete}, \cite{hughes1996ends} and \cite{chodosh2023classifying}.

Ends describe the large-scale shape of a space.
They tell us how many ``directions to infinity'' the space possesses,
and whether these directions eventually merge or remain separated.
While the definition is purely topological, ends will later control
how hypersurfaces can or cannot separate infinity under curvature
constraints.

Let $W$ be a non-compact topological space.

\begin{dfn}
A \emph{neighbourhood of an end} of $W$ is a subspace $U \subset W$ which
contains a connected component of $W \setminus K$ for some non-empty
compact subspace $K \subset W$.
\end{dfn}

\begin{dfn}
An \emph{end} of a non-compact space $W$ is an equivalence class of sequences
of connected open subsets
\[
W \supset U_1 \supset U_2 \supset \cdots
\]
satisfying
\[
\bigcap_{i=1}^{\infty} \overline{U_i} = \varnothing .
\]
Two such sequences $(U_i)$ and $(V_j)$ are equivalent if for every $i$ there
exists $j$ such that $U_i \subset V_j$, and for every $j$ there exists $i$
such that $V_j \subset U_i$.
\end{dfn}

We define the \emph{set of topological ends} of $X$, denoted by $\operatorname{Ends}(X)$, as follows. Consider the directed system of compact sets $K \subset X$ such that $X \setminus K$ has no compact connected component.
For each such $K$, let $\pi_0(X \setminus K)$ be the set of connected components of
$X \setminus K$.
Then $\operatorname{Ends}(X)$ is defined as the inverse limit
\[
\operatorname{Ends}(X) := \varprojlim_{K} \pi_0(X \setminus K).
\]
The \emph{number of topological ends} of $X$ is the cardinality
\[
e(X) := |\operatorname{Ends}(X)|.
\]

A striking feature of ends is that, under mild hypotheses, they become
purely algebraic.
If a group acts cocompactly on a space, the topology at infinity of that
space is completely determined by the group. By \cite{epstein1962ends}, if a group $G$ acts cocompactly by covering
transformations on a locally finite simplicial complex $\bar X$, then the set
$\operatorname{Ends}(\bar X)$, and hence its cardinality $e(\bar X)$, depends only on the group $G$.
In particular, $e(\bar X)$ is a group invariant.

\begin{dfn}[Ends of a finitely generated group]\label{def:ends_group}
Let $G$ be a finitely generated group.
The set of ends of $G$ is defined by
\[
\operatorname{Ends}(G) := \operatorname{Ends}(\bar X),
\]
where $\bar X \to X$ is any regular covering of a finite simplicial complex $X$
with covering transformation group $G$.
The number of ends of $G$ is
\[
e(G) := |\operatorname{Ends}(G)|.
\]
\end{dfn}

In particular, for a finitely generated group $G$, the number $e(G)$ coincides with
the number of topological ends of its Cayley graph.
Similarly, if $M$ is a closed manifold, then
\[
e(\pi_1(M)) = |\operatorname{Ends}(\widetilde M)|,
\]
where $\widetilde M$ denotes the universal cover of $M$.

\begin{rem}\label{rem:ends_classification}
Despite the apparent flexibility of infinite spaces, the number of ends
of a finitely generated group can only be
$0,\;1,\;2,\;\text{or }\infty,
$ (see \cite{epstein1962ends}).

An immediate consequence of the definition is that $e(G)=0$ if and only if $G$ is finite.
Using covering space theory, one shows that $e(G)=2$ if and only if $G$ is virtually
infinite cyclic; that is, $G$ contains an infinite cyclic subgroup of finite index.
\end{rem}

\subsection{Homopoty group and (Co)homology group at infinity}
Ends detect connectedness at infinity, but they miss higher-dimensional
phenomena.
Cohomology at infinity refines this picture by encoding
cohomological classes that remain non-trivial for all sufficiently
large complements of compact sets.
For background on homology and cohomology at infinity and locally finite homology, see {\cite[Chapter~6]{spanier2012algebraic}, \cite[Chapter~5]{bredon1997sheaf}}.

Let $W$ be a non-compact manifold and let
\[
K_1 \subset K_2 \subset \cdots \subset W
\]
be an exhaustion of $W$ by compact subsets, i.e.\ $\bigcup_{j} K_j = W$.
\begin{dfn}\label{def:Homologyatinfinity}
With the same notation, the \emph{homology at infinity} of $W$ is defined by
\begin{equation}
    H_\ast^\infty(W)
\;:=\;
\varprojlim_{j} H_\ast(W \setminus K_j),
\end{equation}
where the inverse limit is taken with respect to the maps induced by
inclusion
\[
H_\ast(W \setminus K_{j+1}) \longrightarrow H_\ast(W \setminus K_j).
\]
This definition is independent of the choice of exhaustion.
\end{dfn}

Dually, cohomology at infinity is obtained as a direct limit.

\begin{dfn}\label{def:Cohomologyatinfinity}
Let $W$ be a non-compact manifold and let
\[
K_1 \subset K_2 \subset \cdots \subset W
\]
be an exhaustion by compact subsets with $W=\bigcup_j K_j$.
The \emph{cohomology at infinity} of $W$ is defined by

\begin{equation}
H^\ast_\infty(W)
\;:=\;
\varinjlim_{j} H^\ast(W \setminus K_j),
\end{equation}
where the direct limit is taken with respect to the restriction maps
\[
H^\ast(W \setminus K_j) \longrightarrow H^\ast(W \setminus K_{j+1}).
\]
Again, this definition does not depend on the choice of exhaustion.
\end{dfn}
On non-compact manifolds, it is often more natural to work with a homology
theory defined directly at the level of chains, allowing infinite chains
that are locally finite.
This leads to the notion of locally finite homology, which is well adapted
to the geometry and topology at infinity.

\begin{dfn}\label{def:Locallyfinitehomology}
The \emph{locally finite homology} of $W$, denoted by $H_\ast^{\mathrm{lf}}(W)$,
is defined as the homology of the chain complex
$C_\ast^{\mathrm{lf}}(W)$ consisting of (possibly infinite) singular chains
\[
c=\sum_i a_i \sigma_i
\]
such that every compact subset of $W$ intersects only finitely many
simplices $\sigma_i$ with non-zero coefficient.
\end{dfn}

The groups $H^\ast_\infty(W)$ encode the topology of $W$ near infinity.
Roughly speaking, a non-trivial class in $H^\ast_\infty(W)$ represents
cohomological information that cannot be killed by removing a compact set.
We record several basic properties and examples that will be used later in Section~\ref{sec:psc_cycle}.

Euclidean space serves as the basic model:
it has exactly one non-trivial class at infinity, corresponding to the
``sphere at infinity''.
\begin{ex}
For $n \ge 2$, the cohomology at infinity of $\mathbb{R}^n$ satisfies
\[
H^k_\infty(\mathbb{R}^n)
\;\cong\;
\begin{cases}
\mathbb{Z}, & k = 0, n-1, \\
0, & \text{otherwise}.
\end{cases}
\]
\end{ex}
Removing a large ball from $\mathbb{R}^n$ yields a space homotopy equivalent
to $S^{n-1}$.
Passing to the direct limit recovers the reduced cohomology of $S^{n-1}$.
\begin{ex}
Let $Y$ be a compact manifold.
Then
\[
H^\ast_\infty(Y \times \mathbb{R}^k)
\;\cong\;
H^{\ast}(Y \times S^{k-1}).
\]
\end{ex}
For tame manifolds, infinity acquires a boundary. In this case, topology at infinity is no longer mysterious—it is encoded entirely by the boundary of a compactification.
If $W$ is tame, so that $W=\operatorname{int}(\overline W)$
for a compact manifold with boundary $\overline W$, then the homology and
cohomology at infinity are completely determined by the boundary
$\partial\overline W$.
\begin{equation}\label{eq:tame_infty}
H^\ast_\infty(W) \;\cong\; H^\ast(\partial\overline W),
\qquad
H^{\infty}_\ast(W) \;\cong\; H_\ast(\partial\overline W).
\end{equation}
This follows from the fact that, for a tame manifold, the structure maps
in the directed system defining are eventually
isomorphisms, so the limit stabilizes.

\begin{lem}\label{lem:lf_rel}
Let $M$ be a non-compact $n$-dimensional manifold. Assume that $M$ is \emph{tame}, i.e.\ there exists a compact manifold $\overline M$ with boundary $\partial \overline M$ such that \[ M \cong \overline M \setminus \partial \overline M . \] Then for every $k$ there is a natural isomorphism \[ H_k^{\mathrm{lf}}(M) \;\cong\; H_k(\overline M,\partial \overline M). \] \end{lem}
\begin{proof}
We argue using Poincaré duality. As $M$ is oriented of dimension $n$, Poincaré duality for non-compact
manifolds yields a natural isomorphism
\[
H_k^{\mathrm{lf}}(M)\cong H^{n-k}(M).
\]

By the tameness assumption, $M=\operatorname{int}(\overline M)$.
Using a collar neighborhood of $\partial\overline M$, we can push
$\overline M$ slightly into $M$ and find a region of $M$ that is
diffeomorphic to $\overline M$.
In particular, $\overline M$ deformation retracts onto its interior $M$,
and hence $M$ and $\overline M$ are homotopy equivalent, and in particular
\[
H^{n-k}(M)\cong H^{n-k}(\overline M).
\]

Finally, since $\overline M$ is a compact oriented manifold with boundary,
Poincaré--Lefschetz duality gives
\[
H^{n-k}(\overline M)\cong H_k(\overline M,\partial\overline M).
\]

Combining these natural isomorphisms, we obtain
\[
H_k^{\mathrm{lf}}(M)\cong H_k(\overline M,\partial\overline M),
\]
as claimed.
\end{proof}
\begin{prop}
Let $M$ be a non-compact $n$-manifold which is \emph{tame}, i.e.\ there exists
a compact manifold $\overline M$ with boundary $\partial \overline M$ such that
\[
M \;\cong\; \operatorname{int}(\overline M)=\overline M\setminus\partial \overline M .
\]
Then there exists a natural homomorphism
\[
\partial_*: H^{\mathrm{lf}}_k(M)\;\longrightarrow\; H^\infty_{k-1}(M).
\]
\end{prop}

\begin{proof}
Since $M$ is tame, locally finite homology of $M$ can be identified
with the relative homology of the compactification $\overline M$.
More precisely, for every $k$ there is a natural isomorphism by Lemma~\ref{lem:lf_rel}
\begin{equation}
    H_k^{\mathrm{lf}}(M)\;\cong\; H_k(\overline M,\partial \overline M).
\end{equation}

Consider the long exact sequence in homology associated with the pair
$(\overline M,\partial \overline M)$:

\begin{equation}
\cdots \longrightarrow
H_k(\overline M)
\longrightarrow
H_k(\overline M,\partial \overline M)
\xrightarrow{\;\partial\;}
H_{k-1}(\partial \overline M)
\longrightarrow
H_{k-1}(\overline M)
\longrightarrow \cdots .
\end{equation}
On the other hand, for a tame manifold $M$ the homology at infinity
is naturally identified with the homology of the boundary (see \eqref{eq:tame_infty}):
\[
H_k^\infty(M)
\;=\;
\varinjlim_{K\Subset M} H_k(X\setminus K)
\;\cong\;
H_k(\partial \overline M).
\]
Indeed, for sufficiently large compact subsets $K\subset M$,
the complement $M\setminus K$ is homotopy equivalent to
$\partial \overline M\times [0,\infty)$.

Combining these identifications, we define

\begin{equation}
\partial_*:
H^{\mathrm{lf}}_k(M)
\;\xrightarrow{\;\cong\;}
H_k(\overline M,\partial \overline M)
\;\xrightarrow{\;\partial\;}
H_{k-1}(\partial \overline M)
\;\xrightarrow{\;\cong\;}
H^\infty_{k-1}(M).
\end{equation}
\end{proof}

\begin{lem}\label{lem:intersection-hypersurface}
Let $X^n$ be a non-compact oriented manifold, and let
$\sigma \subset X$ be a closed curve such that
$[\sigma] \in H_1(X)$ is a non-trivial, non-torsion class.
Then for any precompact open set $D \subset X$ containing a neighborhood
of $\sigma$, there exists a smooth compact $(n-1)$--dimensional submanifold
$M_0 \subset X$ with
\[
\partial M_0 \subset X \setminus D,
\]
such that the algebraic intersection number of $M_0$ and $\sigma$ is equal to $1$.
\end{lem}

\begin{proof}
Since $X$ is a non-compact oriented manifold, Poincar\'e duality provides an
isomorphism
\[
H_1(X) \;\cong\; H^{n-1}_c(X).
\]
It follows that there exists a connected precompact open subset
$A \subset X$ with $D \subset A$ and a cohomology class
\[
\alpha \in H^{n-1}(X, X \setminus A)
\]
satisfying

\begin{equation}
\alpha \frown \mu_A = [\sigma],
\end{equation}
where $\mu_A$ denotes the fundamental class of $A$.

By assumption, the class $[\sigma]$ is non-torsion.
Consequently, $\alpha$ is non-torsion in $H^{n-1}(X, X \setminus A)$.
By the universal coefficient theorem, there exists a class
\[
\beta \in H_{n-1}(X, X \setminus A)
\]
such that
\[
\alpha \frown \beta = 1 \in H_0(X, X \setminus A).
\]

Using excision together with Lefschetz duality, we obtain canonical
identifications

\begin{equation}
H_{n-1}(X, X \setminus A)
\;\cong\;
H^1(A)
\;\cong\;
[A, S^1]_*,
\end{equation}
where $[A,S^1]_*$ denotes the set of basepoint-preserving homotopy
classes of maps from $A$ to $S^1$.

Let $f \colon A \to S^1$ be a smooth representative corresponding to the
class $\beta$, and let $p \in S^1$ be a regular value of $f$.
Then the level set
\[
M_0 := f^{-1}(p) \subset A
\]
is a smooth compact hypersurface whose boundary is contained in
$X \setminus D$, and which represents the homology class $\beta$.
By construction, the algebraic intersection number of $M_0$
with $\sigma$ is equal to $1$.

\end{proof}
\begin{prop}\label{prop:ray-infinity}
Let $W$ be a connected, non-compact oriented $n$--manifold.
Any proper ray
\[
\gamma : [0,\infty) \longrightarrow W
\]
defines a non-trivial class in
\[
\alpha_\gamma \in H^{n-1}_\infty(W).
\]
\end{prop}

\begin{proof}
Fix an exhaustion of $W$ by compact sets
\[
K_1 \subset K_2 \subset \cdots \subset W,
\qquad
\bigcup_{j} K_j = W,
\]
with smooth boundaries.
We now describe the associated cohomology class at infinity.
Fix once and for all a differential form
\[
\omega \in \Omega^{n-1}(W)
\]
constructed from the Thom class of the normal bundle of the ray $\gamma$ (cf. \cite{bott2013differential}).
This closed form represents a candidate class in $H^{n-1}_\infty(W)$.

For each $j$, its restriction
\(
\omega|_{W\setminus K_{j-1}}
\)
defines a cohomology class in $H^{n-1}(W\setminus K_{j-1})$. Consider the boundary
\(
\partial K_j \subset W\setminus K_{j-1}.
\)
Let
\(
\eta_j \in \Omega^1_c(W\setminus K_{j-1})
\)
be a closed differential form representing the Poincar\'e dual of
$[\partial K_j]$ in $W\setminus K_{j-1}$.
By construction, $\eta_j$ has compact support contained in a small collar
neighborhood of $\partial K_j$.

For all sufficiently large $j$, the locally finite cycle $\gamma$ intersects
$\partial K_j$ transversely in an odd number of points.
As a result, the evaluation of the Poincar\'e dual class represented by
$\eta_j$ on $\gamma$ is non-zero.

Equivalently, the pairing

\begin{equation}
\int_{W\setminus K_{j-1}} \omega \wedge \eta_j
\end{equation}
is non-zero.
Since $\eta_j$ is compactly supported, this pairing is well-defined.

It follows that the restriction of $\omega$ to $W\setminus K_{j-1}$
represents a non-trivial cohomology class.
Moreover, these classes are compatible under restriction as $j$ increases,
and hence determine a non-zero element
\[
\alpha_\gamma \in H^{n-1}_\infty(W)
= \varinjlim_j H^{n-1}(W\setminus K_j).
\]
\end{proof}

\begin{rem}
By the proposition above, any proper ray $\gamma \subset W$ determines
a distinguished cohomology class

\begin{equation}
\alpha_\gamma \in H^{n-1}_\infty(W).
\end{equation}
We refer to $\alpha_\gamma$ as the \emph{Poincar\'e dual at infinity}.
\end{rem}

A fundamental result due to Gromov--Lawson asserts that uniform positivity of
scalar curvature imposes strong topological restrictions at infinity.
\begin{dfn}\label{def:sci}
A manifold $M$ is said to be \emph{simply connected at infinity} if for every
compact subset $B \subset M$, there exists a compact subset
$K \subset M$ with $B \subset K$ such that the homomorphism
\[
\pi_1(M \setminus K) \longrightarrow \pi_1(M \setminus B),
\]
induced by the inclusion, is trivial.
\end{dfn}
\begin{thm}[{\cite[Corollary~10.9]{gromov1983positive}}]
\label{thm:gl_sci}
Let $M$ be a complete $3$-manifold of uniformly positive scalar curvature.
If the fundamental group $\pi_1(M)$ is finitely generated, then $M$ is simply
connected at infinity.
\end{thm}
The argument in Theorem~\ref{thm:link_infinity} is similar in spirit to that of Theorem~\ref{thm:gl_sci}.
One considers a suitable exhaustion and performs a mild deformation near the boundary so as to arrange mean convexity.

\begin{rem}\label{rem:fg_assumption}
The finite generation assumption on the fundamental group is essential.
Without it, the conclusion of Theorem~\ref{thm:gl_sci} fails. Consider the $3$-manifold $M$ obtained as an infinite connected sum of copies of $S^2 \times S^1$ constructed from the graph; the half-line
$[0,+\infty)$ with vertices at the integer points.
Then $M$ is not simply connected at infinity, since the complement of
any compact subset contains infinitely many non-contractible loops.
\end{rem}

In the following sections, we explain how positive scalar curvature interacts
with the topology at infinity, sometimes forcing, and sometimes obstructing,
the existence of hypersurfaces that escape to infinity.

\section{Topology at Infinity and Scalar Curvature Decay}\label{sec:sc_decay}
Here, we briefly recall the geometric meaning of the assumptions appearing
in Theorem~\ref{thm:link_infinity}.
We consider a complete oriented Riemannian $4$--manifold $(X,g)$ of
weakly bounded geometry (see Definition~\ref{def:weakly-bounded-geometry}) together with a proper smooth map
$\pi \colon X \to \mathbb{R}^3$.
Let $F=\pi^{-1}(y)$ be the inverse image of a regular value $y$.
Assume that $H_3(X)=0$ and that there exists
\[
0\neq \alpha \in \ker\bigl(H_2^\infty(X)\to H_2(X)\bigr)
\quad\text{with}\quad
\operatorname{Lk}(F,\alpha)\neq 0 .
\]

The linking condition involves a locally finite homology class
$\alpha \in H_2^\infty(X)$ and a regular fiber
$F=\pi^{-1}(y)$ of the proper map $\pi$.
The requirement $\operatorname{Lk}(F,\alpha)\neq 0$ expresses that
the compact fiber $F$ is forced to interact nontrivially with
the homology at infinity of $X$.

More precisely, recall that the homology at infinity
$H_2^\infty(X)$ is defined as the inverse limit
\[
H_2^\infty(X)
\;=\;
\varprojlim_K H_2(X \setminus K),
\]
where $K$ ranges over compact subsets of $X$.
Thus a class $\alpha \in H_2^\infty(X)$ is represented by a compatible
family of homology classes
\[
\alpha_K \in H_2(X \setminus K)
\]
for all sufficiently large compact sets $K$.

The assumption
\[
\alpha \in \ker\bigl(H_2^\infty(X)\to H_2(X)\bigr)
\]
means that, for every sufficiently large $K$, the class $\alpha_K$
bounds in $X$, although it remains nontrivial in $H_2(X \setminus K)$.
Since $H_3(X)=0$, we may choose, for each such $K$, a $3$--chain
$W_K \subset X$ satisfying
\[
\partial W_K = \alpha_K .
\]

Let $F=\pi^{-1}(y)$ be a regular fiber.
For $K$ large enough so that $F \subset K$, the usual intersection
pairing defines a linking number
\[
\operatorname{Lk}_K(F,\alpha)
\;:=\;
\langle F, [W_K] \rangle ,
\]
which is well-defined because $F$ is disjoint from $\alpha_K$.
As $K$ increases, these linking numbers are compatible under restriction
and hence stabilize.
We define the linking number at infinity by the stabilized value
$\operatorname{Lk}(F,\alpha)
\;:=\;
\lim_{K \to \infty} \operatorname{Lk}_K(F,\alpha).$
Consequently, $\operatorname{Lk}(F,\alpha)$ is well-defined and depends
only on the class $\alpha \in H_2^\infty(X)$.

Geometrically, the condition $\operatorname{Lk}(F,\alpha)\neq 0$
means that the fiber $F$ nontrivially links the $2$--dimensional
topology at infinity.
In particular, $F$ cannot be separated from the end of $X$
represented by $\alpha$ by any compact deformation.
This linking at infinity provides the topological input that leads
to the scalar curvature estimate in
Theorem~\ref{thm:link_infinity}.
\begin{thm}
\label{thm:link_infinity}
Let $(X^4,g)$ be a complete oriented Riemannian manifold with weakly bounded geometry
and non-negative scalar curvature.
Let $\pi \colon X \to \mathbb{R}^3$ be a proper smooth map,
and let $F = \pi^{-1}(y)$ be the inverse image of a regular value.
Assume that $H_3(X)=0$ and that there exists
\begin{equation}
0 \neq \alpha \in \ker\bigl(H_2^\infty(X) \to H_2(X)\bigr)
\quad \text{with} \quad
\operatorname{Lk}(F,\alpha) \neq 0 .
\label{eq:alpha-linking-assumption}
\end{equation}
For $\rho>0$, set
$
\pi_\rho := \pi\big|_{\pi^{-1}(B^3(\rho))} \colon \pi^{-1}(B^3(\rho)) \to B^3(\rho).$
Then there exists a constant $C > 0$, depending only on the dimension,
such that for all $\rho > 0$,
\begin{equation}
\min_{\pi^{-1}(B^3(\rho))} R_g
\;\le\;
C
\left(\frac{\operatorname{Dil}_1(\pi_\rho)}{\rho}\right)^2 .
\end{equation}
\end{thm}

\begin{proof}[Proof of Theorem~\ref{thm:link_infinity}]
\noindent\textbf{Step 1.}
Since
\begin{equation*}
0 \neq \alpha \in \ker\bigl(H_2^\infty(X)\to H_2(X)\bigr),
\end{equation*}
by definition there exist nested bounded open sets
\begin{equation*}
\Omega_1 \subset \Omega_2 \subset \cdots \subset X,
\qquad \bigcup_{j=1}^\infty \Omega_j = X,
\end{equation*}
and nontrivial classes
\begin{equation}\label{eq:alpha-j}
\alpha_j \in H_2(X\setminus \Omega_j)
\end{equation}
such that for all $i\le j$ the inclusion
\(
\iota : X\setminus \Omega_j \hookrightarrow X\setminus \Omega_i
\)
satisfies $\iota_*(\alpha_j)=\alpha_i\neq 0$, while the image of $\alpha_j$
in $H_2(X)$ vanishes.

Consequently, for each $j$ there exists a smooth embedded $2$--cycle
\begin{equation*}
\Sigma_j \subset X\setminus \Omega_j
\end{equation*}
representing $\alpha_j$ and null-homologous in $X$.

\medskip
\noindent\textbf{Step 2.}
Let $F:=\pi^{-1}(pt)$ be a regular fiber.
Since $\operatorname{Lk}(F,\alpha)\neq 0$, for all sufficiently large $j$ we have
\begin{equation}\label{eq:linking-Sigma-j}
\operatorname{Lk}(F,\Sigma_j)\neq 0.
\end{equation}
Hence any $3$--chain $\widetilde M$ with $\partial\widetilde M=\Sigma_j$
must intersect $F$.

We modify the metric near $\partial\Omega_j$ to obtain a metric $g_j$
agreeing with $g$ away from $\partial\Omega_j$ and such that
$\partial\Omega_j$ is mean convex (see similar arguments in {\cite[Corollary~10.9]{gromov1983positive}}, {\cite[Theorem~1.13]{chodosh2024complete}}).
Consider the Plateau problem
\begin{equation}\label{eq:Plateau-problem}
\inf\bigl\{\mathcal{H}^3_{g_j}(\widetilde M)\;:\;
\partial\widetilde M=\Sigma_j\bigr\}.
\end{equation}
By standard existence and regularity theory,
there exists a smooth compact two-sided area-minimizing hypersurface
$M_j$ with $\partial M_j=\Sigma_j$.
By the linking property, $M_j\cap F\neq\emptyset$.

\medskip
\noindent\textbf{Step 3.}
Each $M_j$ is stable.
Lemma~\ref{lem:stable-curvature-estimate} yields uniform curvature bounds
on compact subsets of $X$.
After passing to a subsequence,
\begin{equation}\label{eq:Mj-convergence}
M_j \to M_\infty
\end{equation}
where
\begin{equation}\label{eq:Minfty-immersion}
M_\infty^3 \looparrowright (X^4,g)
\end{equation}
is a complete, two-sided, stable minimal immersion.
Moreover, $M_\infty\cap F\neq\emptyset$, so $M_\infty$ is nonempty.

\medskip
\noindent\textbf{Step 4.}
Stability of $M_\infty$ implies the inequality
\begin{equation}\label{eq:stability-inequality}
\int_{M_\infty} \bigl(2|\nabla f|^2 + R_{M_\infty} f^2\bigr)
\;\ge\;
\int_{M_\infty} \bigl(R_X + |A|^2\bigr) f^2
\end{equation}
for all $f\in C_c^\infty(M_\infty)$.
By the Fischer--Colbrie--Schoen argument (\cite{fischer1980structure}), there exists a positive function
$u>0$ on $M_\infty$ satisfying
\begin{equation}\label{eq:FCS-equation}
-2\Delta u + R_{M_\infty} u = (R_X + |A|^2)u.
\end{equation}

Consider the warped product
\begin{equation*}
\widehat M := M_\infty \times S^1
\end{equation*}
equipped with the metric
\begin{equation*}
\tilde g := g_{M_\infty} + u^2\, dt^2.
\end{equation*}
A direct computation yields
\begin{equation}\label{eq:scalar-curvature-warped}
R_{\tilde g}
=
R_{M_\infty} - \frac{2\Delta u}{u}
=
R_X + |A|^2.
\end{equation}

Fix a degree--one map
\begin{equation}\label{eq:phi-map}
\varphi \colon \mathbb{R}^3 \longrightarrow S^3
\end{equation}
which is constant outside the unit ball and has uniformly bounded
Lipschitz constant.
For each $\rho>0$, consider the composition
\begin{equation}\label{eq:composition-map}
\widehat M
\;\xrightarrow{\;\pi|_{M_\infty}\times \mathrm{id}\;}
\mathbb{R}^3 \times S^1
\;\xrightarrow{\;\varphi_\rho \times \mathrm{id}\;}
S^3 \times S^1,
\end{equation}
where $\varphi_\rho(x) := \varphi(x/\rho)$.
This map is $O\!\left(\operatorname{Dil}_1(\pi)/\rho\right)$--contracting.

We note that $M_\infty$ admits a spin structure since it is a $3$-manifold, which implies that $\widehat M$ is also spin.
Passing to a sufficiently high finite covering of $\widehat M$,
we obtain a non-zero degree map
\begin{equation}\label{eq:degree-one-map}
F \colon \widehat M \longrightarrow S^4
\end{equation}
which is contracting by a factor comparable to
$\operatorname{Dil}_1(\pi_\rho)/\rho$ and is constant outside a compact set.
Applying the same argument based on the relative index theorem
\cite[Theorem~11.4]{gromov1983positive},
the existence of such a contracting non-zero degree map implies the estimate
\begin{equation}\label{eq:final-estimate}
\min_{\pi^{-1}(B^3(\rho))} R_g
\;\le\;
\mathrm{const}_3
\left(\frac{\operatorname{Dil}_1(\pi_\rho)}{\rho}\right)^2 .
\end{equation}
For completeness, we include a proof.
Without loss of generality, we may assume the following setting, since the relevant geometric and analytic estimates are stable under taking the product with a sufficiently large two–sphere, and no essential information is lost as the radius tends to infinity.

Suppose that there exists a non-compact Riemannian spin manifold $N^4$
equipped with a complete Riemannian metric $g_N$ such that
\begin{equation}
    R(g_N) \ge \kappa > 0
\quad \text{on } N \setminus K,
\end{equation}

where $K \subset N$ is a sufficiently large compact set.
Assume moreover that there is a map $G : N \to S^4$
with the same contraction factor as $F$, and that a subset
$\Omega_\rho \subset K$ is identified by an isometry with
\begin{equation}
    \Omega_\rho \cong \pi^{-1}(B^3(\rho)) \times S^1 .
\end{equation}
Under this identification, the maps
\begin{equation}
    F : (\widehat M, \pi^{-1}(B^3(\rho)) \times S^1) \to (S^4, \mathrm{pt}),
\qquad
G : (N, \Omega_\rho) \to (S^4, \mathrm{pt})
\end{equation}
are compatible.
Choose a complex vector bundle $E_0$ over $S^{4}$
such that
$
\widetilde{\mathrm{ch}}(E_0)\neq 0 .
$
Fix a unitary connection $\nabla^0$ on $E_0$, and denote its curvature by $R^0$.
Let $G : N \to S^{4}$ be the map constructed above.
Define
\begin{equation*}
    E := G^*E_0, \qquad \nabla := G^*\nabla^0 .
\end{equation*}

Since $G$ is constant outside the compact set $\Omega_\rho\subset N$,
the bundle $E$ is flat on $N \setminus \Omega_\rho$, and in particular
\begin{equation}\label{eq:G_flat_outside}
\mathcal{R}^E = 0
\quad \text{on } N \setminus \Omega_\rho.
\end{equation}

Moreover, since $G$ has the same contraction factor as $F$,
the curvature of $E$ satisfies the pointwise estimate
\begin{equation}\label{eq:G_curvature_bound}
\|\mathcal{R}^E\| \le C\,\|\mathcal{R}^{E_0}\|(\frac{\operatorname{Dil}_1(\pi_\rho)}{\rho})^2
\quad \text{on } \Omega_\rho,
\end{equation}
where we set $C$ is a constant depending only on demension.
Let $\mathbb{S}$ be the spinor bundle over $N$, and let
\[
\slashed D_E : \Gamma(\mathbb{S}\otimes E) \to \Gamma(\mathbb{S}\otimes E)
\]
be the Dirac operator twisted by $E$.

We consider two elliptic operators on $N$:
the untwisted Dirac operator $\slashed D$ acting on the spinor bundle $\mathbb{S}$,
and the Dirac operator $\slashed D_E$ acting on $\mathbb{S}\otimes E$.

By \eqref{eq:G_flat_outside}, \eqref{eq:G_curvature_bound}, and the assumption
that $R(g_N)\ge \kappa>0$ on $N\setminus K$ with $\Omega_\rho\subset K$,
the Weitzenböck curvature term associated to both operators is uniformly
positive outside a compact set.
Hence $\slashed D$ and $\slashed D_E$ are coercive at infinity and therefore Fredholm.

Since $E$ is trivial and the connection is flat on $N\setminus\Omega_\rho$,
the operators $\slashed D$ and $\slashed D_E$ agree outside a compact set.
Thus the relative index theorem applies and yields
\[
\operatorname{ind}(\slashed D_E)-\operatorname{ind}(\slashed D)
=
\int_N (\mathrm{ch}(E)-\mathrm{rk}(E))\,\hat A(N)
=
\deg(G)\int_{S^{4}}\widetilde{\mathrm{ch}}(E_0),
\]
which is non-zero by construction.

On the other hand, the Weitzenböck formula for the untwisted operator $\slashed D$
shows $\slashed D$ admits no non-trivial $L^2$ harmonic spinors and
$
\operatorname{ind}(\slashed D)=0.
$

It follows that $\operatorname{ind}(\slashed D_E)\neq 0$, and hence $\slashed D_E$ has a
non-trivial kernel.
In particular, the Weitzenböck curvature term for $\slashed D_E$ must become negative at some point of $\Omega_\rho$.
\end{proof}
\begin{rem}\label{rem:spin-bernstein-only}
We emphasize that the existence of a spin structure is an essential input
in the above argument, as it is required for the application of the
Dirac operator and the relative index theorem.
In the proof presented above, this condition was automatically satisfied
due to the \emph{dimension--specific fact that any orientable 3--manifold is spin}.

Apart from this spin assumption, however, the proof only uses
\begin{itemize}
  \item curvature estimates for complete stable minimal hypersurfaces, and
  \item the classification result on the Bernstein problem asserting that
  complete stable minimal hypersurfaces in Euclidean space must be flat.
\end{itemize}
In particular, no further global classification results are invoked.
As a consequence, the same method applies in spin manifold settings
where these analytic ingredients and the above classification result
are available, yielding analogous scalar curvature decay estimates
under suitable linking assumptions at infinity.
Indeed, if the ambient manifold $X$ is spin, then any two--sided
hypersurface inherits a canonical spin structure, since its normal
bundle is trivial.
\end{rem}

The scalar curvature estimate established in
Theorem~\ref{thm:link_infinity}
is driven by two large--scale features of a noncompact manifold:
the existence of nontrivial topology at infinity,
and the presence of a proper map to Euclidean space
whose distortion can be quantitatively controlled.
In particular, the existence of a proper map with controlled dilation
allows one to convert nontrivial linking at infinity
into quantitative curvature bounds.
This motivates the following definitions.

\begin{dfn}[Linking degree at infinity]
Let $X^n$ be an oriented manifold with $H_{n-1}(X)=0$,
and let
\(
\pi \colon X \to \mathbb R^{n-1}
\)
be a proper smooth map.

The \emph{linking degree at infinity} of $\pi$ is defined by
\[
\deg_\infty(\pi)
:=
\sup
\Bigl\{
\operatorname{Lk}(\pi^{-1}(y),\beta)
\;\Big|\;
y \text{ is a regular value of } \pi,\;
\beta \in \mathcal K
\Bigr\},
\]
where
\[
\mathcal K
:=
\ker\bigl(H_{n-2}^\infty(X)\to H_{n-2}(X)\bigr).
\]
If $\mathcal K=\{0\}$, we set $\deg_\infty(\pi):=0$.
\end{dfn}
\begin{dfn}[Weak domination over Euclidean space]\label{def:weak_domination}
Let $(X^n,g)$ be a complete oriented Riemannian manifold
with $H_{n-1}(X)=0$.
We say that $(X,g)$ is ($\alpha$--)\emph{weakly dominated over $\mathbb R^{n-1}$}
if there exists a proper smooth map
\[
\pi \colon X \to \mathbb R^{n-1}
\]
such that:

\begin{enumerate}
\item[(i)] (\emph{Nontrivial domination at infinity})
\[
\deg_\infty(\pi)\neq 0 .
\]

\item[(ii)] (\emph{Polynomial dilation growth})
There exist constants $C>0$ and $0< \alpha\le 1$ such that
for all $\rho>0$,
\[
\mathrm{Dil}_1(\pi_\rho)\le C\,\rho^{1-\alpha} .
\]
\end{enumerate}
\end{dfn}
Combining Theorem~\ref{thm:link_infinity}
with the definition of $\alpha$--weak domination over $\mathbb R^3$,
we obtain a quantitative scalar curvature decay estimate.

\begin{cor}\label{cor:thm:link_infinity}
Let $(X^4,g)$ be a complete oriented Riemannian manifold
with weakly bounded geometry and non-negative scalar curvature.
If $(X,g)$ is $\alpha$--weakly dominated over $\mathbb R^3$ by a map
\(
\pi \colon X \to \mathbb R^3,
\)
then for all $\rho>0$,

\[
\min_{\pi^{-1}(B^3(\rho))} R_g
\;\lesssim\;
\rho^{-2\alpha} .
\]

\end{cor}

\begin{ex}[Quadratic scalar curvature decay in a weakly dominated manifold]
\label{ex:gromov_quadratic_decay}

We illustrate the class of weakly dominated manifolds
by a concrete example originating from a construction of Gromov.
In \cite{gromov2018metric, chen2025optimal},
a complete Riemannian metric with positive scalar curvature
and quadratic decay at infinity was constructed on
$\mathbb R^2\times T^2$.

Passing to the infinite cyclic cover in one of the $T^2$--directions,
we obtain a complete oriented $4$--manifold
\[
X \cong \mathbb R^2 \times \mathbb R \times S^1
\]
with weakly bounded geometry and non-negative scalar curvature.
Moreover, $H_3(X)=0$, and $X$ admits a proper smooth map
\[
\pi \colon X \to \mathbb R^3
\]
with nonzero linking degree at infinity.
In particular, $(X,g)$ is weakly dominated over $\mathbb R^3$
in the sense of Definition~\ref{def:weak_domination}.

Applying Corollary~\ref{cor:thm:link_infinity} to the dominating map $\pi$,
we obtain an upper bound on the scalar curvature in terms of the
$1$--dilation of $\pi$.
Using the explicit dilation estimate
\eqref{eq:dilation-growth},
\[
\mathrm{Dil}_1(\pi_\rho)
\sim
\frac{1}{c}\,\rho^{1-\alpha},
\qquad
0<\alpha\le \tfrac14,
\]
the corollary implies that the scalar curvature must decay at least
at a polynomial rate as $\rho\to\infty$.

For the specific metric under consideration,
a direct computation (cf.\ \eqref{eq:scalar-final})
shows that
\[
\min_{\pi^{-1}(B^3(\rho))} R_g \sim \frac{C}{\rho^2},
\]
that is, the decay is quadratic.
Although Corollary~\ref{cor:thm:link_infinity} does not recover
the optimal decay exponent in this example,
it nevertheless detects nontrivial scalar curvature decay
purely from the topology at infinity and the large--scale behavior
of the dominating map.
Indeed,
\[
\min_{\pi^{-1}(B^3(\rho))} R_g
\sim
\frac{1}{\rho^2}
\;\lesssim\;
\left(\frac{\mathrm{Dil}_1(\pi_\rho)}{\rho}\right)^2
\sim
\rho^{-2\alpha}.
\]

The detailed construction of the metric,
the scalar curvature computation,
and the matching argument at the gluing radius
are presented in Appendix~\ref{sec:construction_qdecay}.
\end{ex}

Second, the same argument can be read in the opposite direction to yield
a priori estimates for the minimal hypersurfaces themselves.
When the ambient manifold has non-negative scalar curvature,
the stability inequality places strong restrictions on how much curvature
a stable minimal hypersurface can accumulate far away from a fixed base
point.
The resulting statement is closely analogous to
\cite[Theorem~11.4]{gromov1983positive}.

\begin{cor}
Assume that $X$ is tame and that $\pi$ is distance--decreasing.
Let $g$ be a complete Riemannian metric on $M = X \times S^1$
with weakly bounded geometry and non-negative scalar curvature.
Let $\alpha \in H_2^\infty(X)\cong H_2(\partial X)$ be a non-trivial
homology class at infinity.

Assume that there exists a complete, stable minimal hypersurface
$\Sigma^3 \subset M$ representing a locally finite homology class
$[\Sigma]\in H_3^{\mathrm{lf}}(M)$ whose boundary at infinity satisfies
$\partial_*([\Sigma])=\alpha$.
Then $M$ satisfies the decay estimate
\[
\min_{p\in \pi^{-1}(B_r)\cap M} |A_M|(p)
\;\le\;
\frac{C}{r(p)}
\]
where $r(p)$ denotes the distance
from $p$ to a fixed base point in $M$.
\end{cor}

\section{Small Tori and Obstructions to Positive Scalar Curvature}\label{sec:small_torus}

\subsection{The 3--Dimensional Picture}
In this subsection, we briefly recall several results on positive scalar curvature
in dimension three, following \cite{gromov1983positive}.
These results will serve as a conceptual model for the higher--dimensional
constructions introduced later. Throughout this section, all manifolds are assumed to be oriented, unless stated otherwise.

In dimension three, positive scalar curvature imposes very strong topological
constraints. In particular, certain embedded surfaces and curves act as
obstructions to the existence of complete metrics of uniformly positive scalar
curvature. Understanding these obstructions will motivate the definition of
\emph{small tori} in higher dimensions.

For compact $3$--manifolds, the classical Kneser--Milnor prime decomposition theorem
\cite{kneser1929geschlossene,milnor1962unique}, together with Perelman’s resolution
of the Poincar\'e conjecture, gives a complete topological classification.
Every closed orientable $3$--manifold admits a unique decomposition
\[
M = X_1 \# \cdots \# X_l \# (S^2 \times S^1) \# \cdots \# (S^2 \times S^1)
   \# K_1 \# \cdots \# K_m,
\]
where the $X_i$ are spherical manifolds and the $K_j$ are aspherical.

For non--compact $3$--manifolds, the situation is subtler.
Gromov and Lawson \cite{gromov1983positive} introduced special classes of surfaces
and curves which obstruct the existence of complete metrics of uniformly positive
scalar curvature. These objects capture essential topological information at
infinity and play a central role in our discussion.

We begin with the notion of taut surfaces, which may be viewed as a
non--compact generalization of incompressible surfaces.
\begin{dfn}[Taut and incompressible surfaces]\label{def:taut_incompressible}
Let $X$ be a manifold.
A surface $S \subset X$ is said to be \emph{taut} if it is properly embedded,
has infinite fundamental group, and the induced homomorphism
\[
\pi_1(S) \longrightarrow \pi_1(X)
\]
is injective. A taut surface is called \emph{incompressible} if it is compact.
\end{dfn}

Every incompressible surface is taut.  
For example, the surface $S \times \{0\} \subset S \times \mathbb{R}$ (with $S$ compact)
is incompressible, while the surface
\[
S^1 \times \mathbb{R} \subset S^1 \times \mathbb{R}^2
\]
is taut.

We recall the following obstruction to positive scalar curvature due to
Gromov--Lawson \cite{gromov1983positive}.
\begin{thm}[Gromov--Lawson {\cite[Theorem~8.4]{gromov1983positive}}]
\label{thm:taut_surface_obstruction}
Let $X$ be an open $3$-manifold.
If $X$ contains a taut surface, then $X$ admits no complete metric of
uniformly positive scalar curvature.
\end{thm}
\begin{rem}
If $X$ satisfies the hypothesis of
Theorem~\ref{thm:taut_surface_obstruction}, then so does the connected sum
$X \# Y$ for any $3$-manifold $Y$.
\end{rem}

The obstruction provided by taut surfaces admits an equivalent formulation in
terms of embedded curves, known as \emph{small circles}.
This reformulation will be particularly convenient for our purposes, since it
suggests a natural higher--dimensional analogue.
The obstruction in Theorem~\ref{thm:taut_surface_obstruction} admits an
alternative formulation, also due to Gromov--Lawson.

\begin{dfn}
\label{def:small_circle}
Let $X$ be a $3$-manifold.
A smoothly embedded circle $\gamma \subset X$ is called \emph{small} if
\begin{enumerate}
\item $\gamma$ has infinite order in $H_1(X)$, and
\item the normal circle to $\gamma$ represents an element of infinite order in
$H_1(X \setminus \gamma)$.
\end{enumerate}
\end{dfn}

\begin{thm}[Gromov--Lawson {\cite[Theorem~8.7]{gromov1983positive}}]
\label{thm:small_circle}
Let $X$ be an open $3$-manifold with $H_1(X)$ finitely generated.
If $X$ contains a small circle, then $X$ admits no complete metric of
uniformly positive scalar curvature.
\end{thm}

\subsection{Small Tori in Higher Dimensions}

The notion of a small circle captures a topological feature that is invisible to
local geometry but cannot coexist with uniformly positive scalar curvature.
Our goal is to isolate an analogous structure in higher dimensions.
\begin{dfn}[Small torus]\label{def:small_torus}
Let $X$ be an $n$--dimensional manifold with $H_1(X)$ finitely generated.
Suppose that there exists a smooth embedding
\[
\iota \colon T^{n-2} \hookrightarrow X
\]
whose normal bundle is trivial. Let $Z$ denote the unit normal circle bundle of $\iota(T^{n-2})$.
We say that $\iota(T^{n-2})$ is a \emph{small torus} in $X$ if the embedding
$\iota$ induces an injective homomorphism
\[
H_1(Z) \longrightarrow H_1\bigl(X \setminus \iota(T^{n-2})\bigr).
\]
\end{dfn}
\begin{rem}\label{rem:small_torus_circle}
When $n=3$, the torus $T^{n-2}$ reduces to a circle, and
Definition~\ref{def:small_torus} recovers the notion of a
small circle introduced by Gromov--Lawson
(Definition~8.6 in \cite{gromov1983positive}).
\end{rem}

Our main result shows that the presence of a small torus obstructs the existence
of complete metrics of uniformly positive scalar curvature, extending the
Gromov--Lawson obstruction from dimension three to higher dimensions.
\begin{thm}\label{thm:small_torus_obstruction}
For $n \leq 7$, let $X$ be an open $n$--manifold with $H_1(X)$ finitely generated.
If $X$ contains a small torus, then $X$ admits no complete metric of
uniformly positive scalar curvature.
Moreover, there exists an end $X_+ \subset X$ such that there is no
complete Riemannian metric $g$ on $X$ satisfying
\[
R_g \ge \sigma > 0 \quad \text{on } X_+ .
\]
\end{thm}

\begin{rem}\label{rem:comparison_chodosh_li}
Compared with the theorem resolving the generalized Geroch
conjecture with arbitrary ends~\cite[Theorem~3]{chodosh2024generalized},
which operates under the existence of a suitable compactification at infinity,
Theorem~\ref{thm:small_torus_obstruction} yields a strictly weaker obstruction,
especially for manifolds of the form $T^n \# X$.
Indeed, our approach detects obstructions on suitable
non-compact coverings of $M = T^n \# X$ via the presence of a small torus.
For instance, consider a surjection $\pi_1(M)\to\mathbb Z^2$.
Then the associated covering $\widetilde M$ contains a small torus.
While this leads to a weaker conclusion, it is conceptually closer
to the classical small circle obstruction and is flexible enough
to apply in more general non-compact settings without assuming
compactifiability at infinity.
\end{rem}

\begin{proof}[Proof of Theorem~\ref{thm:small_torus_obstruction}]
Let $Z \subset X$ be a small torus.
Denote by $i : Z \hookrightarrow X \setminus \iota(T^{n-2})$ the inclusion.
Suppose by contrary that $X$ admits a complete metric of uniformly positive scalar
curvature.

\medskip
\noindent\textbf{Step~1:}
By assumption, $H_1(X)$ is finitely generated.
Since $X \setminus \iota(T^{n-2})$ is obtained from $X$ by removing a samll torus,
$H_1(X \setminus \iota(T^{n-2}))$ is a finitely generated abelian group.
Tensoring with $\mathbb{Q}$, we obtain an inclusion of finite--dimensional
$\mathbb{Q}$--vector spaces
\begin{equation*}
i_* \otimes \mathrm{id}_{\mathbb{Q}} :
H_1(Z;\mathbb{Q}) \hookrightarrow
H_1(X \setminus \iota(T^{n-2});\mathbb{Q}).
\end{equation*}
Since $H_1(Z;\mathbb{Q}) \cong \mathbb{Q}^{n-2}$ and the injectivity of $i_* \otimes \mathrm{id}_{\mathbb{Q}}$ (Definition~\ref{def:small_torus}),
we may choose a $\mathbb{Q}$--linear projection
\begin{equation*}
p_{\mathbb{Q}} :
H_1(X \setminus \iota(T^{n-2});\mathbb{Q})
\longrightarrow H_1(Z;\mathbb{Q})
\end{equation*}
such that
$p_{\mathbb{Q}} \circ (i_* \otimes \mathrm{id}_{\mathbb{Q}})
= \mathrm{id}.
$
Multiplying by a suitable integer clears denominators and yields
a homomorphism
\begin{equation}\label{eq:projection-Z}
p : H_1(X \setminus \iota(T^{n-2}))
\longrightarrow H_1(Z) \cong \mathbb{Z}^{n-2}
\end{equation}
such that the composition
$
p \circ i_* : H_1(Z) \longrightarrow H_1(Z)
$
is injective.

\medskip
\noindent\textbf{Step~2.}
Set
$
Y := X \setminus \nu\bigl(\iota(T^{n-2})\bigr).
$
Since $Z=T^{n-2}$ is an Eilenberg--MacLane space
$K(\mathbb Z^{n-2},1)$, there is a natural identification
\begin{equation}\label{eq:YZ-identification}
[Y,Z]
\;\cong\;
H^1(Y;\mathbb Z^{\,n-2})
\;\cong\;
\operatorname{Hom}\bigl(H_1(Y),\mathbb Z^{\,n-2}\bigr).
\end{equation}
Hence the homomorphism
\begin{equation}\label{eq:p-homology}
p:H_1(Y)\longrightarrow H_1(Z)\cong\mathbb Z^{\,n-2}
\end{equation}
is realized by a continuous map
$
f:Y\longrightarrow Z,
$
which is unique up to homotopy.

\medskip
Moreover, by construction the induced map
\begin{equation}\label{eq:f-induced}
f_*|_{H_1(Z)} = p\circ i_* :
H_1(Z)\longrightarrow H_1(Z)
\end{equation}
is injective.
Since $H_1(Z)$ is free abelian, the universal coefficient theorem yields
\begin{equation*}
H^1(Z;\mathbb Z)\cong \mathrm{Hom}(H_1(Z),\mathbb Z),
\end{equation*}
and hence
\begin{equation}\label{eq:fZ-surjective}
(f|_Z)^*:H^1(Z;\mathbb Z)\to H^1(Z;\mathbb Z)
\end{equation}
is surjective.

Since $Z$ is an $(n-2)$--torus, we have
\begin{equation}\label{eq:top-cohomology-torus}
H^{n-2}(Z;\mathbb Q)
\cong \bigwedge^{n-2} H^1(Z;\mathbb Q)
\cong \mathbb Q.
\end{equation}
Taking exterior products, the surjectivity of $(f|_Z)^*$ on $H^1(Z)$ implies that
\begin{equation}\label{eq:top-degree-map}
(f|_Z)^*:
H^{n-2}(Z;\mathbb Q)
\longrightarrow H^{n-2}(Z;\mathbb Q)
\end{equation}
is a non-zero homomorphism.
Equivalently,
\begin{equation*}
(f|_Z)^*([Z])=\deg(f|_Z)\,[Z]
\end{equation*}
with $\deg(f|_Z)\neq 0$.

\medskip
\noindent\textbf{Step~3:}
The existence of a non-zero degree self-map of $Z$ will lead to a contradiction
once combined with the enlargeability of $Z$.
To make this precise, we use the following lemma, which treats a slightly more
general situation. We postpone the proof of the following lemma.
\end{proof}

\begin{lem}\label{lem:psc-degree-obstruction}
Let $(X,g)$ be a non-compact complete Riemannian manifold, and let
$Z \subset X$ be a closed hypersurface.
Suppose that $X_+ \subset X$ is a non-compact connected component with
$\partial X_+ = Z$, and that there exists a continuous map
\[
f \colon X_+ \to Z
\]
whose restriction $f|_Z \colon Z \to Z$ has non-zero degree. Assume that $X_+$ has uniformly positive scalar curvature $R_g\ge\sigma>0$ on $X_+$.

Then there exists a family of compact, connected, separating hypersurfaces
$\Sigma_j \subset X_+$ $(j \in \mathbb{N})$ with spectrally positive scalar curvature
such that the following holds:
\begin{enumerate}
  \item There exists at least one index set $J \subset \mathbb{N}$ such that
  \begin{equation}\label{eq:lf-hlogy}
    \sum_{j \in J} [\Sigma_j] = [Z]
    \quad \text{in } H^{\mathrm{lf}}_{n-1}(X).
  \end{equation}

  \item For any such index set $J$, there exists at least one $j \in J$
  such that the restricted map
  \begin{equation}\label{eq:degree-on-Sigma}
    f|_{\Sigma_j} \colon \Sigma_j \longrightarrow Z
  \end{equation}
  has non-zero degree.
\end{enumerate}

\end{lem}
\begin{proof}[Proof of Theorem~\ref{thm:small_torus_obstruction} (continued)]
First, we fix some non-compact component $X_+\subset Y$ whose boundary is $Z$. By Lemma~\ref{lem:psc-degree-obstruction}, there exists a compact separating
hypersurface $\Sigma\subset X_+$ with spectrally positive scalar curvature such
that the restriction
\begin{equation*}
f|_\Sigma:\Sigma\longrightarrow Z
\end{equation*}
has non-zero degree.
Since $Z$ is enlargeable, the existence of such a hypersurface contradicts the
assumption that $X$ admits a complete metric of uniformly positive scalar
curvature.
This completes the proof.
\end{proof}
\begin{proof}[Proof of Lemma~\ref{lem:psc-degree-obstruction}]
Assume that $X$ admits a complete metric of uniformly positive scalar
curvature.
By the $\mu$-bubble exhaustion (Theorem~\ref{thm:exhaustion}),
there exists an exhaustion of $X$ by compact domains
\begin{equation*}
K_1 \subset K_2 \subset \cdots \subset X
\end{equation*}
such that each boundary $\partial K_i$ is a smooth, two-sided hypersurface
with spectrally positive scalar curvature.
Moreover, for $i$ sufficiently large, $\partial K_i$ separates $X$
into an interior compact region and an exterior end.

Fix $i$ large enough so that some connected component of $\partial K_i$
is contained in the region $X_+$.
Since each $\partial K_i$ separates, we obtain an identity in locally finite
homology
\begin{equation}\label{eq:lf-homology-identity}
\sum_j [\Sigma_j]=[Z]
\quad\text{in } H_{n-1}^{\mathrm{lf}}(X),
\end{equation}
where each $\Sigma_j$ is a connected component of $\partial K_i$.

We now reformulate the argument at the level of locally finite homology.
By a small homotopy of $f$ on the target, we may replace $Z$ by a one-sided collar neighborhood $Z\times[0,1]$. Here and in the sequel, we identify the zero slice $Z\times\{0\}$ with $Z$.
(Strictly speaking, we replace $f$ by
$F(p)=(f(p),\varphi\circ h(p))$, where $h$ is a smooth regularization of
$\operatorname{dist}(\,\cdot\,,Z)$ and $\varphi:\mathbb{R}_{\ge 0}\to[0,1)$ is a
diffeomorphism; by abuse of notation we continue to denote this map by $f$.)
We regard $f$ as a proper map
\begin{equation*}
f:X_+\longrightarrow Z\times[0,1],
\end{equation*}
which induces a homomorphism on locally finite homology
\begin{equation}\label{eq:f-induced-lf}
f_*:H_{n-1}^{\mathrm{lf}}(X_+)\longrightarrow
H_{n-1}^{\mathrm{lf}}(Z\times[0,1]).
\end{equation}

Recall that in $H_{n-1}^{\mathrm{lf}}(X_+)$ we have the identity
\begin{equation}\label{eq:lf-identity-recall}
\sum_j [\Sigma_j]=[Z].
\end{equation}
Suppose, to the contrary, that
\begin{equation}\label{eq:fSigma-zero}
f_*([\Sigma_j])=0
\quad\text{in } H_{n-1}^{\mathrm{lf}}(Z\times[0,1])
\end{equation}
for every $j$.
By linearity of $f_*$, this implies
\begin{equation}\label{eq:fZ-zero}
f_*([Z])=\sum_j f_*([\Sigma_j])=0
\quad\text{in } H_{n-1}^{\mathrm{lf}}(Z\times[0,1]).
\end{equation}

However, the class $[Z]$ is represented by a compact hypersurface contained
in $Z\times\{0\}$, and hence defines a non-trivial class in
\begin{equation}\label{eq:Z-homology}
H_{n-1}(Z)\cong H_{n-1}^{\mathrm{lf}}(Z\times[0,1]).
\end{equation}
This contradiction shows that at least one of the hypersurfaces $\Sigma_j$
satisfies
\begin{equation}\label{eq:fSigma-nonzero}
f_*([\Sigma_j])\neq 0
\quad\text{in } H_{n-1}(Z).
\end{equation}
We note that the above argument does not depend on any special choice of the
index set $J$. Once an index set $J \subset \mathbb{N}$ satisfies
\eqref{eq:lf-hlogy}, the same reasoning applies verbatim.
In particular, for any such $J$, there exists at least one $j \in J$
for which the restricted map
\[
f|_{\Sigma_j} \colon \Sigma_j \to Z
\]
has non-zero degree.
\end{proof}

\section{Positive scalar curvature and cycles at infinity}\label{sec:psc_cycle}
First, we recall the notion of Schoen--Yau--Schick manifolds (SYS manifolds),
following definition in \cite{schick1998counterexample,schoen1987structure}.

\begin{dfn}\label{def:sys_mfd}
Let $n \ge 2$.
A compact orientable $n$--manifold $M$ is called a
\emph{Schoen--Yau--Schick manifold} (or \emph{SYS manifold}) if there exist
homology classes
\[
h_1,\ldots,h_{n-2} \in H_1(M;\mathbb{Z})
\]
such that the $2$--dimensional homology class
\[
\sigma := h_1 \frown \cdots \frown h_{n-2} \frown [M]
\in H_2(M;\mathbb{Z})
\]
is \emph{non-spherical}, that is, $\sigma$ does not lie in the image of the
Hurewicz homomorphism
\[
\pi_2(M) \longrightarrow H_2(M;\mathbb{Z}) .
\]
\end{dfn}

A basic example of an SYS manifold is the torus.
Using minimal surface techniques together with an inductive descent
argument, Schoen and Yau \cite{schoen1987structure} proved that SYS manifolds
of dimension at most $7$ do not admit Riemannian metrics of positive scalar
curvature.
Later, Schick \cite{schick1998counterexample} constructed an SYS manifold
that provides a counterexample to the unstable
Gromov--Lawson--Rosenberg conjecture.

\begin{thm}\label{thm:Ray_obstruction}
Let $n \le 7$.
Let $X^n$ be a connected oriented non-compact manifold, and let
$\varepsilon$ be an isolated end of $X$.
Assume that there exists a proper ray
$\gamma \subset \varepsilon$
whose Poincar\'e dual at infinity
\[
\alpha_\gamma \in H^{n-1}_\infty(\varepsilon)
\]
admits a decomposition
\[
\alpha_\gamma
=
u^1 \cup u^2 \cup \cdots \cup u^{n-1},
\qquad
u^i \in H^1_\infty(\varepsilon).
\]
Then there exists no complete Riemannian metric $g$ on $\varepsilon$
satisfying
\[
R_g \ge \sigma > 0
\quad \text{on } \varepsilon .
\]

\end{thm}
\begin{ex}
First, let us consider
\[
M = \mathbb{CP}^2 \# T^4,
\]
and let
$
\pi \colon \widetilde M \to M
$
be the infinite cyclic cover associated with a primitive (i.e.\ indivisible) cohomology class
$
\phi \in H^1(M;\mathbb Z)
$
coming from one of the $S^1$--factors of $T^4$.
Then $\widetilde M$ is a connected, oriented, non-compact $4$--manifold with
exactly two isolated ends.
(The primitivity of $\phi$ implies surjectivity onto $\mathbb Z$, and hence connectedness of the cover.)
Topologically, each end consists of an infinite chain of copies of
$\mathbb{CP}^2 \setminus B^4$,
attached along $3$--spheres to
$\mathbb{R} \times T^3$.

In what follows, all non-compact manifolds constructed using the ends described above
provide concrete examples to which Theorem~\ref{thm:Ray_obstruction} applies.
More precisely, each such manifold contains an isolated end admitting a proper ray
whose Poincar\'e dual at infinity decomposes as a cup product of degree--one classes (In this case, the end cohomology at infinity is of rank one, so the Poincar\'e dual
automatically admits the required decomposition, and Theorem~\ref{thm:Ray_obstruction}
applies
).
Therefore it cannot carry a complete Riemannian metric of uniformly positive scalar curvature.

\end{ex}

\begin{proof}[Proof of Theorem~\ref{thm:Ray_obstruction}]
We argue by contradiction.
Assume that $X$ admits a complete Riemannian metric $g$
with uniformly positive scalar curvature.

\medskip
\noindent\textbf{Step~1.}
We begin by unpacking the definition of cohomology at infinity.
Choose an exhaustion of $X$ by relatively compact subsets
\begin{equation*}
K_1 \subset K_2 \subset \cdots \subset X .
\end{equation*}
By definition,
\begin{equation*}
H^{n-1}_\infty(X;\mathbb{R})
\;=\;
\varinjlim_j H^{n-1}(X \setminus K_j;\mathbb{R}) .
\end{equation*}
Hence the class $\alpha_\gamma \in H^{n-1}_\infty(X;\mathbb{R})$
is represented by a cohomology class
\begin{equation}\label{eq:alpha-gamma-j}
\alpha_{\gamma,j} \in H^{n-1}(X \setminus K_j;\mathbb{R})
\end{equation}
for all sufficiently large $j$, and these representatives are compatible
under restriction.
In particular, for $j$ large enough, we may write
\begin{equation}\label{eq:alpha-wedge}
\alpha_{\gamma,j}
=
u_j^1 \wedge \cdots \wedge u_j^{n-1}
\quad \text{in } H^{n-1}(X \setminus K_j;\mathbb{R}),
\end{equation}
where each $u_j^k \in H^1(X \setminus K_j;\mathbb{R})$.
This simply reflects the fact that we are unpacking the inductive
limit description of the cohomology class at infinity.

\medskip
\noindent\textbf{Step~2.}
We now use the cohomology class constructed in Step~1 to produce a
nontrivial hypersurface at infinity.
For $j$ sufficiently large, consider the class
\begin{equation}\label{eq:alpha-epsilon}
\alpha_{\gamma,j}
=
u_j^1 \wedge \cdots \wedge u_j^{n-1}
\in H^{n-1}(\varepsilon \setminus K_j;\mathbb{R}) .
\end{equation}
By Poincar\'e duality on $\varepsilon \setminus K_j$, this class evaluates
nontrivially on some $(n-1)$--dimensional homology class.
Hence we can find a connected, properly embedded hypersurface
\begin{equation*}
\Sigma \subset \varepsilon \setminus K_j
\end{equation*}
such that
\begin{equation}\label{eq:pairing-nonzero}
\langle \alpha_{\gamma,j}, [\Sigma] \rangle \neq 0 .
\end{equation}
Equivalently, $\Sigma$ represents a nontrivial homology class in
$H_{n-1}(\varepsilon \setminus K_j)$ that is detected by the cohomology
class at infinity.
In particular, $\Sigma$ is a Schoen--Yau--Schick (SYS) hypersurface
in the sense of Definition~\ref{def:sys_mfd}.

By an argument analogous to that of Lemma~\ref{lem:intersection-hypersurface}, there exists a proper ray
at infinity
\begin{equation}\label{eq:gamma-Sigma}
\gamma_\Sigma \subset \varepsilon \setminus K_j
\end{equation}
such that the algebraic intersection number satisfies
\begin{equation}\label{eq:intersection-number}
\gamma_\Sigma \cdot \Sigma = 1 .
\end{equation}
Since $\varepsilon$ is an isolated end of $X$, the proper ray
$\gamma_\Sigma$ represents the same locally finite homology class as
$\gamma$, that is,
\begin{equation}\label{eq:lf-homology}
[\gamma] = [\gamma_\Sigma]
\in H_1^{\mathrm{lf}}(\varepsilon \setminus K_j) .
\end{equation}
Consequently, the associated cohomology classes (defined in Proposition~\ref{prop:ray-infinity}) coincide ,
\begin{equation}\label{eq:alpha-coincide}
\alpha_\gamma = \alpha_{\gamma_\Sigma}
\in H^{n-1}(\varepsilon \setminus K_j) .
\end{equation}

\medskip
\noindent\textbf{Step~3.}
On the other hand, the argument can be localized entirely to a sufficiently
far band inside the end $\varepsilon$.
Fix a compact set $K_j \subset X$ and consider a compact domain
$
M \subset \varepsilon \setminus K_j
$
whose boundary decomposes as
$
\partial M = \partial^- M \sqcup \partial^+ M,
$
with both components non-empty such that the distance between the
two boundary components satisfies
$
d(\partial^- M, \partial^+ M) > D(\sigma),
$
where $D(\sigma)$ is the constant from
Proposition~\ref{prop:existence_band}.
By Proposition~\ref{prop:existence_band}, there exists a smoothly embedded,
closed, two-sided hypersurface
$
\Sigma^{n-1} \subset \operatorname{int}(M) \subset \varepsilon \setminus K_j
$
which admits a Riemannian metric of positive scalar curvature.
By construction, $\Sigma$ separates the same homology at infinity as the chosen
cross section of the end, and hence represents the same class in
$
H_{n-1}(\varepsilon \setminus K_j).
$

However, these two hypersurfaces have incompatible geometric properties.
By construction, each $\partial \Omega_i$ has spectrally positive scalar
curvature and hence admits a conformal metric of positive scalar curvature.
Recall that $\Sigma$ is a Schoen--Yau--Schick (SYS) hypersurface
(Definition~\ref{def:sys_mfd}), and therefore admits no metric of
positive scalar curvature.
Since $\partial \Omega_i$ and $\Sigma$ represent the same homology class
in $H_{n-1}(\varepsilon \setminus K_j)$ for $i$ sufficiently large,
this yields a contradiction.
\end{proof}

\appendix
\section{Weakly bounded geometry}
\label{subsec:weakly-bounded-geometry}\label{sec:weakly_bdd}

The proof strategy in this section is based on ideas developed in \cite{chodosh2024complete}.
While the presentation is adapted to our setting and notation, the overall strategy follows their approach.

\begin{dfn}[Weakly bounded geometry]
\label{def:weakly-bounded-geometry}
Let $(X^n,g)$ be a complete Riemannian manifold.
We say that $(X,g)$ has \emph{$Q$--weakly bounded geometry} if for every
$p\in X$ there exists a $C^{2,\alpha}$ local diffeomorphism
\[
\Phi \colon (B(0,Q^{-1}),0)\subset \mathbb{R}^n \to (U,p)\subset X
\]
such that
\begin{enumerate}
  \item
  the pullback metric satisfies
  \[
  e^{-2Q}\,\delta \le \Phi^* g \le e^{2Q}\,\delta
  \quad\text{as bilinear forms,}
  \]
  \item
  the metric coefficients satisfy
  \[
  \|\partial \Phi^* g\|_{C^{\alpha}} \le Q .
  \]
\end{enumerate}
We say that $(X,g)$ has \emph{weakly bounded geometry} if this holds for some
$Q<\infty$.
\end{dfn}

\begin{rem}
This condition allows for collapsing behavior at infinity.
For instance, hyperbolic cusps have weakly bounded geometry, even though their
injectivity radius tends to zero.
It is well known that a uniform sectional curvature bound
$|\sec_g|\le K$ implies weakly bounded geometry for some $Q=Q(K)$.
Alternatively, bounds of the form $|\operatorname{Ric}_g|\le K$ together with a uniform lower
injectivity radius bound also suffice.
\end{rem}

We now record a curvature estimate {\cite[Lemma~2.4]{chodosh2024complete}} for stable minimal hypersurfaces in
$4$--manifolds with weakly bounded geometry.

\begin{lem}
\label{lem:stable-curvature-estimate}
Let $(X^4,g)$ be a complete Riemannian $4$--manifold with $Q$--weakly bounded
geometry.
Then there exists a constant $C=C(Q)<\infty$ such that every compact two-sided
stable minimal immersion
\[
M^3 \to (X^4,g)
\]
satisfies
\[
\sup_{q\in M} |A_M(q)| \,\min\{1,d_M(q,\partial M)\} \le C .
\]
\end{lem}

\begin{proof}
Suppose the conclusion fails.
Then there exists a sequence of compact two-sided stable minimal immersions
$M_i^3 \to (X_i^4,g_i)$ such that
\begin{equation}
\sup_{q\in M_i} |A_{M_i}(q)| \,\min\{1,d_{M_i}(q,\partial M_i)\} \to \infty .
\end{equation}
Since the function on the left is continuous and vanishes on $\partial M_i$,
there exists a point $p_i\in M_i\setminus\partial M_i$ achieving the supremum.
Set
\[
r_i := |A_{M_i}(p_i)|^{-1} \to 0 ,
\]
and let $x_i$ denote the image of $p_i$ in $X_i$.

By the weakly bounded geometry assumption, there exist local diffeomorphisms
\[
\Phi_i \colon (B(0,Q^{-1}),0)\subset \mathbb{R}^4 \to (X_i,x_i)
\]
with uniformly controlled metric coefficients.
Using the lifting construction (cf.{\cite[Appendix~C]{chodosh2024complete}}), we obtain
pointed manifolds $(S_i,s_i)$, immersions

\begin{equation}
F_i \colon (S_i,s_i) \to (B(0,Q^{-1}),0),
\end{equation}
and local diffeomorphisms
\[
\Psi_i \colon (S_i,s_i) \to (M_i,p_i),
\]
such that $F_i$ is a two-sided stable minimal immersion with respect to
$\Phi_i^* g_i$.

Define the rescaling maps
\[
D_i(x)= r_i x ,
\]
and consider the rescaled metrics
\[
\tilde g_i := r_i^{-2} D_i^* \Phi_i^* g_i
\]
on $B(0,r_i^{-1}Q^{-1})$.
By construction, the metrics $\tilde g_i$ converge in $C^{1,\alpha}_{\mathrm{loc}}$ to
the Euclidean metric on $\mathbb{R}^4$.

Let
\[
\tilde F_i := D_i^{-1}\circ F_i .
\]
Then $\tilde F_i$ is a stable minimal immersion with respect to $\tilde g_i$.
The point-picking construction implies that for every fixed $\rho<\infty$,
there exists $C(\rho)$ such that

\begin{equation}
|A_{\tilde F_i}(q)| \le C(\rho)
\quad\text{whenever } d(q,s_i)\le \rho .
\end{equation}
Standard graphical arguments and elliptic estimates yield uniform local
$C^{2,\alpha}$ bounds for $\tilde F_i$.
In particular, the induced metrics $\tilde h_i := \tilde F_i^*\tilde g_i$ have
uniformly bounded curvature and injectivity radius on compact subsets.

By pointed Cheeger--Gromov compactness, a subsequence converges to a complete
Riemannian manifold $(S,h,s)$, and the immersions $\tilde F_i$ converge in
$C^{2,\beta}_{\mathrm{loc}}$ (for any $\beta<\alpha$) to a limiting immersion

\begin{equation}
F \colon (S,s) \to (\mathbb{R}^4,0).
\end{equation}
The limit immersion is complete, two-sided, stable, minimal, and satisfies
$|A_F(s)|=1$.
Due to the theorem of \cite{chodosh2023stable}, such an immersion must be
flat, a contradiction.
This completes the proof.
\end{proof}

\section{A model metric with quadratic scalar curvature decay}\label{sec:construction_qdecay}
\subsection{Computation of the scalar curvature}

First, we briefly review the construction of a complete smooth Riemannian metric $g$ on $\mathbb R^2\times T^{n-2}$
with positive scalar curvature and quadratic decay at infinity in \cite{gromov2018metric, chen2025optimal}.

Let $n\ge 3$.
We work on $\mathbb R^2\times T^{n-2}$ with polar coordinates $(r,\theta)$ on
$\mathbb R^2$ and flat metric $g_{T^{n-2}}$ on $T^{n-2}$.
Fix constants $c>0$ and $0<r_0<\frac{\pi}{2}$.

\medskip

For $r\ge r_0$, consider the warped product metric
\begin{equation}\label{eq:gplus}
g_+
=
dr^2
+
c^2 r^{2\alpha}
\bigl(
d\theta^2 + g_{T^{n-2}}
\bigr),
\qquad \alpha>0.
\end{equation}
Let $f(r)=c r^\alpha$ be the warping function.
Since $d\theta^2+g_{T^{n-2}}$ is flat, the scalar curvature of $g_+$ is given by
\begin{equation}\label{eq:scalar-general}
R_{g_+}
=
- (n-1)(n-2)\frac{(f')^2}{f^2}
- 2(n-1)\frac{f''}{f}.
\end{equation}
For $f(r)=c r^\alpha$, we have
$
f'(r)=c\alpha r^{\alpha-1},
\qquad
f''(r)=c\alpha(\alpha-1)r^{\alpha-2}.
$
Substituting into \eqref{eq:scalar-general}, we obtain
\begin{align}
R_{g_+}
&=
- (n-1)(n-2)\frac{\alpha^2}{r^2}
- 2(n-1)\frac{\alpha(\alpha-1)}{r^2}
\notag\\
&=
\frac{1}{r^2}
\Bigl(
- n(n-1)\alpha^2 + 2(n-1)\alpha
\Bigr)
\notag\\
&=
\frac{1}{r^2}
\left(
- n(n-1)\Bigl(\alpha-\frac1n\Bigr)^2
+ \frac{n-1}{n}
\right).
\label{eq:scalar-final}
\end{align}
In particular, if $0<\alpha\le \frac1n$, then $R_{g_+}>0$ and for $x=(r,\theta,\varphi_1,\varphi_2)$,
\[
 R_g(x)
\sim
\frac{1}{\rho^2}
\]
i.e.\ the scalar curvature has quadratic decay.

\medskip

The mean curvature of the hypersurface
$\{r=r_0\}\times S^1\times T^{n-2}$
with respect to the outward unit normal $\partial_r$ is
\begin{equation}\label{eq:Hplus}
H_+
=
(n-1)\frac{f'(r_0)}{f(r_0)}
=
\frac{(n-1)\alpha}{r_0}.
\end{equation}

\medskip

For $r\le 2r_0$, we consider the product-type metric
\begin{equation}\label{eq:gminus}
g_-
=
dr^2
+
\sin^2(r)\,d\theta^2
+
\sin^2(r_0)\, g_{T^{n-2}}.
\end{equation}
Since $dr^2+\sin^2(r)d\theta^2$ is the standard metric on $S^2$,
the metric $g_-$ is smooth at $r=0$ and has scalar curvature
$
R_{g_-}=2.
$

The mean curvature of
$\{r=r_0\}\times S^1\times T^{n-2}$
with respect to $\partial_r$ is
\begin{equation}\label{eq:Hminus}
H_-
=
\frac{\cos(r_0)}{\sin(r_0)}.
\end{equation}
For $r_0>0$ sufficiently small, we have
$
H_- > H_+.
$
Choosing
\[
c=\frac{\sin(r_0)}{r_0^\alpha},
\]
the metrics $g_-$ and $g_+$ agree at $r=r_0$.
By a smoothing argument,
the metric can be smoothed near $r=r_0$ to obtain a complete
smooth Riemannian metric $g$ on $\mathbb R^2\times T^{n-2}$
with positive scalar curvature and quadratic decay at infinity.
\subsection{Quantitative dilation estimates}

We consider the complete Riemannian manifold
\[
X = \widetilde{\mathbb R^2 \times T^2}
\cong \mathbb R^2 \times \mathbb R \times S^1,
\]
obtained by taking the infinite cyclic cover in one of the $T^2$--directions.

Let $(r,\theta)$ be polar coordinates on $\mathbb R^2$.
Write coordinates on $T^2$ as $(\varphi_1,\varphi_2)\in S^1\times S^1$,
and take the covering in the $\varphi_1$--direction, so that
$\varphi_1\in\mathbb R$ and $\varphi_2\in S^1$.

For $r\ge r_0$, the metric takes the warped product form
\begin{equation}\label{eq:metric}
g
=
dr^2
+
c^2 r^{2\alpha}
\bigl(
d\theta^2 + d\varphi_1^2 + d\varphi_2^2
\bigr),
\qquad 0<\alpha\le \frac14.
\end{equation}

We define a proper smooth map
\[
\pi : X \to \mathbb R^3,
\qquad
\pi(r,\theta,\varphi_1,\varphi_2)
=
\bigl(
r\cos\theta,\,
r\sin\theta,\,
\varphi_1
\bigr).
\]
This map is invariant under rotations in the $\theta$--variable and
forgets the $\varphi_2$--direction.
Recall that the $1$--dilation of $\pi$ at $x=(r,\theta,\varphi_1,\varphi_2)$ is given by
\[
\mathrm{dil}_1(\pi)_x
=
\sup_{0\neq v\in T_xX}
\frac{|d\pi(v)|_{\mathbb R^3}}{|v|_g}.
\]
For
$v
=
a\,\partial_r
+
b\,\partial_\theta
+
u\,\partial_{\varphi_1}
+
w\,\partial_{\varphi_2}$,
\[
|v|_g^2
=
a^2
+
c^2 r^{2\alpha}
(b^2+u^2+w^2),
\qquad
|d\pi(v)|^2
=
a^2+r^2b^2+u^2.
\]

Discarding the $\partial_{\varphi_2}$--component (which only increases the
denominator), we obtain
\[
\frac{|d\pi(v)|^2}{|v|_g^2}
\le
\frac{a^2 + r^2 b^2 + u^2}
     {a^2 + c^2 r^{2\alpha}(b^2+u^2)}.
\]
Maximizing over $(a,b,u)\neq 0$ yields
\[
\mathrm{dil}_1(\pi)_x
=
\max\left\{
1,\,
\frac{r}{c r^\alpha},\,
\frac{1}{c r^\alpha}
\right\}
=
\max\left\{
1,\,
\frac{1}{c} r^{1-\alpha}
\right\}.
\]
In particular, for $\rho\to\infty$,
\begin{equation}\label{eq:dilation-growth}
\mathrm{Dil}_1(\pi_\rho)
\sim
\frac{1}{c}\, \rho^{1-\alpha}.
\end{equation}
where we set
$
\pi_\rho := \pi\big|_{\pi^{-1}(B^3(\rho))} \colon \pi^{-1}(B^3(\rho)) \to B^3(\rho).$

Combining \eqref{eq:dilation-growth} with $\min_{\pi^{-1}(B^3(\rho))} R_g
\sim
\frac{1}{\rho^2}$ shows that
\[
\min_{\pi^{-1}(B^3(\rho))} R_g
\sim
\frac{1}{\rho^2}
\;\lesssim\;
\left(\frac{\mathrm{Dil}_1(\pi_\rho)}{\rho}\right)^2
\sim
\rho^{-2\alpha}.
\]

\section{A quantitative perspective on the obstruction}
\label{sec:appendix_quantitative}

In this appendix, we complement the qualitative obstructions developed in Section~\ref{sec:psc_cycle} by analyzing them from a quantitative perspective. Specifically, we show that the uniform positivity of scalar curvature enforces a universal lower bound on the $1$--dilation of proper maps into spaces with cylindrical ends.
\begin{thm}
\label{thm:quantitative_sys_at_infinity}
Let $f \colon X^5 \to Y^5$ be a proper smooth map of non-zero degree between
complete Riemannian manifolds.
Assume that $X$ has uniformly positive scalar curvature
$R_{g_X} \ge \sigma > 0$.
Assume that $Y$ has a cylindrical end
\[
Y = N \cup (\partial N \times \mathbb{R}),
\]
where $N$ is a compact manifold, $\partial N \cong S^3 \times S^1$ and
$\partial N \times \mathbb{R}$ is isometric to
$S^3 \times S^1 \times \mathbb{R}$ with the standard product metric.

Then the $1$--dilation of $f$ satisfies
\[
\operatorname{Dil}_1(f) \;\ge\; \sqrt{\frac{\sigma}{6}} .
\]
\end{thm}

\begin{proof}
We divide the proof into cases and begin with the product situation.

\medskip
\noindent\textbf{Case~1.}
\medskip
\noindent\textbf{Step~1.}
Assume that the target manifold is of the form
\[
Y = P \times \mathbb{R},
\]
where $P$ is a closed oriented manifold.
We write the given proper map $f : X \to Y$ in components as
\begin{equation}
    f = (f_1,f_2), \qquad 
f_1 : X \to P, \quad f_2 : X \to \mathbb{R}.
\end{equation}

\medskip
We fix a scale parameter $D>0$, which will be taken arbitrarily large.
By Theorem~\ref{thm:exhaustion}, applied to $(X,g)$ with scalar curvature
$R_X \ge \sigma > 0$, we obtain an exhaustion of $X$ by compact domains
whose boundaries form a family of smooth hypersurfaces indexed by $i\in\mathbb{Z}_{\ge 1}$
\[
\Sigma_i := \partial \Omega_i,
\]
satisfying the following properties:
the hypersurfaces $\Sigma_i$ are pairwise separated by distance at least $D$,
and each $\Sigma_i$ has positive scalar curvature in the spectral sense,
that is,

\begin{equation}\label{eq:spc_mu_bubble}
\lambda_1(L_{\Sigma_i}) \ge \delta(D) > 0.
\end{equation}

\medskip
\noindent\textbf{Step~2.}
Next, we cut $X$ by a regular level of the second component $f_2$.
Choose a regular value $t_0 \in \mathbb{R}$ of $f_2$ and set

\begin{equation}
\Sigma_0 := f_2^{-1}(t_0).
\end{equation}
Then $\Sigma_0$ is a smooth, compact, oriented hypersurface in $X$.

Let $\omega_P$ be a normalized volume form on $P$, and let
$\omega_{\mathbb{R}}$ denote the standard orientation form on $\mathbb{R}$.
Since $f$ is proper, the degree of $f$ is well-defined and can be expressed as
\[
\deg f
= \int_X f^*(\omega_P \wedge \omega_{\mathbb{R}})
= \int_X f_1^*(\omega_P) \wedge f_2^*(\omega_{\mathbb{R}}).
\]

By the choice of the regular value $t_0$ and the coarea formula, this integral
reduces to an integral over the fiber $\Sigma_0$:
\[
\deg f
= \int_{\Sigma_0} f_1^*(\omega_P)
= \deg\bigl(f_1|_{\Sigma_0}\bigr).
\]
In particular, the restriction of $f_1$ to $\Sigma_0$ has non-zero degree.

\medskip
Since $\Sigma_0$ is compact, it is contained in $\Omega_i$ for all $i$
sufficiently large.
Fix such an index $i$ and set $\Sigma_1 := \partial \Omega_i$.
Then $\Sigma_0$ and $\Sigma_1$ bound a compact region in $X$ and are therefore
homologous as oriented hypersurfaces.
Because degree is invariant under homology, it follows that
\[
\deg\bigl(f_1|_{\Sigma_1}\bigr)
= \deg\bigl(f_1|_{\Sigma_0}\bigr)
\neq 0.
\]

\medskip
Consequently, setting
\[
F := f_1|_{\Sigma_1} : \Sigma_1 \to P,
\]
we obtain a smooth map from a hypersurface $\Sigma_1$ arising from the
$\mu$--bubble exhaustion, whose degree is non-zero.

\medskip
\noindent\textbf{Step~3.}
We have obtained a hypersurface
\[
\Sigma := \Sigma_1 \subset X
\]
from the $\mu$--bubble exhaustion, together with a smooth map
\[
F := f_1|_{\Sigma} : \Sigma \to S^3 \times S^1
\]
of non-zero degree.
Write $F=(F^{(3)},F^{(1)})$, where
\[
F^{(3)} : \Sigma \to S^3,
\qquad
F^{(1)} : \Sigma \to S^1 .
\]

Consider a generic fiber of $F^{(3)}$, and let
$\gamma \subset \Sigma$ be a loop whose homotopy class represents such a fiber.
Passing to the covering $\widehat\Sigma \to \Sigma$ corresponding to the subgroup
of $\pi_1(\Sigma)$ generated by $F_*[\gamma]$, we may assume that the induced map
\[
\widehat F : \widehat\Sigma \to S^3 \times S^1
\]
induces a surjection on fundamental groups.

Since the universal cover of $S^1$ is $\mathbb{R}$, the map $\widehat F$ lifts to
a proper map
\[
\widetilde F = (\widetilde F^{(3)},\widetilde F^{(1)})
: \widetilde\Sigma \to S^3 \times \mathbb{R}
\]
of non-zero degree.
The dilation of $\widetilde F^{(3)}$ coincides with that of $F^{(3)}$, since the
covering transformations act only on the second factor.
Recall that the stability estimate \eqref{eq:spc_mu_bubble} holds for $\Sigma=\Sigma_1$.
Let $\phi_1>0$ be the first eigenfunction of $L_\Sigma$ such that
$L_\Sigma \phi_1 = \lambda_1(L_\Sigma) \phi_1 .
$
We consider the conformally related metric
\[
\bar g := \phi_1^{\frac{4}{n-2}}\, g_\Sigma
\]
on $\Sigma$, where $n=\dim\Sigma=4$.

By the conformal scalar curvature formula, we have
\[
R_{\bar g}
= \phi_1^{-\frac{n+2}{n-2}}
\bigl(
- c_n \Delta_\Sigma \phi_1 + R_\Sigma \phi_1
\bigr),
\qquad
c_n=\frac{4(n-1)}{n-2}.
\]
Using the eigenvalue equation for $\phi_1$ and the lower
bound $R_X\ge\sigma>0$, we obtain the pointwise inequality
\begin{equation}\label{eq:psc_conformal}
R_{\bar g}
\ge \lambda_1(L_\Sigma)\,\phi_1^{-\frac{4}{n-2}}
\ge \delta(D)\,\phi_1^{-2}
>0.
\end{equation}

In particular, $(\Sigma,\bar g)$ has positive scalar curvature.
Passing to the covering $\widetilde\Sigma\to\Sigma$ constructed above,
we lift $\phi_1$ and $\bar g$ to $\widetilde\Sigma$.
By abuse of notation, we denote the lifted objects again by
$\phi_1$ and $\bar g$.
Since the covering map is a local isometry, the inequality
\eqref{eq:psc_conformal} continues to hold on $\widetilde\Sigma$, and hence
\[
R_{\bar g} > 0
\quad\text{on }\widetilde\Sigma .
\]

\medskip
\noindent\textbf{Step~4.}
We now follow the general strategy of Theorem~2.5 in
\cite{cecchini2021enlargeable}, using $\mu$--bubbles in place of minimal
hypersurfaces.

Let
\[
\Gamma_0 \subset \widetilde\Sigma
\]
be a generic fiber of the second component $\widetilde F^{(1)}$.
Fix $L>0$ large and define the metric neighborhood
\[
U_L := \{ x \in \widetilde\Sigma \mid \operatorname{dist}(x,\Gamma_0) \le L \}.
\]
Applying Proposition~\ref{prop:existence_band} to $(U_L,h)$, we obtain a hypersurface
\[
\Gamma \subset U_L
\]
which is spectrally positive and homologous to $\Gamma_0$.
There exists a positive function $\phi_2$ on $\Gamma$ such that
\begin{equation}
L_{\Gamma}\phi_2
\ge \bigl( \delta(D)\phi_1^{-2} - \pi^2/L^2 \bigr)\phi_2
\quad \text{on } \Gamma .
\end{equation}
Define the conformally related metric
\[
h_\Gamma := \phi_2^{4}\, \bar{g}|_{\Gamma}.
\]
Its scalar curvature satisfies
\begin{equation}\label{eq:scalar_conformal_refresh}
R(h_\Gamma)
= \phi_2^{-5} L_{\Gamma}(\phi_2)
\ge \bigl( \delta(D)\phi_1^{-2} - \pi^2/L^2 \bigr)\phi_2^{-4}.
\end{equation}

Let
\[
G := \widetilde F^{(3)}|_{\Gamma} : \Gamma \to S^3 .
\]
With respect to $h_\Gamma$, the map $G$ is
$(l^2\phi_1^{-2}\phi_2^{-4})$--area--contracting, where
$l:=\operatorname{Dil}(f)$.
Fix $\rho>0$ and let $S^1_\rho$ denote the circle of radius $\rho$
with metric $dt_\rho^2$.
Let
\[
\sigma : S^3 \times S^1 \to S^4
\]
be a $1$--contracting smashing map of non-zero degree, and define
\[
H : \Gamma \times S^1_\rho \to S^4,
\qquad
H := \sigma \circ (G \times \mathrm{id}_\rho).
\]

Equip $\Gamma \times S^1_\rho$ with the product metric
\[
h_4 := h_\Gamma + dt_\rho^2 .
\]
With respect to the product metric $h_4$, the map $H$ is
$(l^2\phi_1^{-2}\phi_2^{-4})$--area--contracting on vectors tangent to $\Gamma$,
and $(1/\rho)$--contracting on vectors tangent to the $S^1_\rho$ factor.
Since the circle factor is flat, the scalar curvature of $h_4$ satisfies
\begin{equation}\label{eq:scalar_product}
R(h_4)
\ge
\bigl( \delta(D)\phi_1^{-2} - \pi^2/L^2 \bigr)\phi_2^{-4} > 0 .
\end{equation}

\medskip
\noindent\textbf{Step~5.}
We now invoke the Gromov--Lawson method~\cite{gromov1983positive}.
Since $\Gamma$ is $3$--dimensional, it admits a spin structure, and hence
$\Gamma \times S^1_\rho$ is spin.
Fix such a spin structure and let $\mathbb{S}$ denote the associated complex
spinor bundle with Dirac operator $\slashed D$.

Following Llarull, we consider the canonical Clifford module over $S^4$.
Define
\[
E_0 := P_{\mathrm{Spin}_4}(S^4) \times_\lambda \mathrm{Cl}_4 ,
\]
where $\lambda$ denotes left Clifford multiplication.
The bundle $E_0$ is equipped with the metric and connection
$\nabla^{E_0}$ induced by the round metric on $S^4$.

Pulling back via the map $H : \Gamma \times S^1_\rho \to S^4$, we obtain a
vector bundle
\[
E := H^* E_0
\]
over $\Gamma \times S^1_\rho$, endowed with the induced connection
$\nabla := H^* \nabla^{E_0}$.
Let
\[
\slashed D_E : \Gamma(\mathbb{S} \otimes E)
\longrightarrow
\Gamma(\mathbb{S} \otimes E)
\]
denote the corresponding twisted Dirac operator.

The volume form on $S^4$ induces a $\mathbb{Z}_2$--grading
$E_0 = E_0^+ \oplus E_0^-$.
Let $\slashed D_{E^+}$ denote the restriction of $\slashed D_E$ to
$\Gamma(\mathbb{S} \otimes E^+)$.
Since $\Gamma \times S^1_\rho$ is even dimensional, the spinor bundle
decomposes as $\mathbb{S} = \mathbb{S}^+ \oplus \mathbb{S}^-$, under which
$\slashed D_{E^+}$ is odd.
Its index is given by
\[
\operatorname{ind}(\slashed D_{E^+})
:= \dim \ker \slashed D_{E^+}^+
- \dim \ker \slashed D_{E^+}^- .
\]

Since $c_2(E_0^+) \neq 0$ and $\deg(H) \neq 0$, while
$p_1(\Gamma \times S^1_\rho)=0$ so that the total $\hat A$--genus equals
$1 \in H^0(\Gamma \times S^1_\rho;\mathbb{Q})$,
the Atiyah--Singer index theorem yields
\[
\operatorname{ind}(\slashed D_{E^+}) \neq 0 .
\]
The computation is identical to that appearing in
\cite{cecchini2021enlargeable}.

\medskip
\noindent\textbf{Step~6.}
We now analyze the kernel of $\slashed D_E^2$.
Let $u \in \Gamma(\mathbb{S} \otimes E)$.
By the Weitzenb\"ock formula,
\begin{equation}\label{eq:bochner_refresh}
\langle \slashed D_E^2 u, u \rangle
\ge \frac14 \langle R(h_4) u, u \rangle
+ \langle \mathcal R_E u, u \rangle ,
\end{equation}
where $\mathcal R_E \in \operatorname{End}(\mathbb{S} \otimes E)$ depends
linearly on the curvature of $\nabla^E$.

Using \eqref{eq:scalar_product}, we obtain the pointwise bound
\begin{equation}\label{eq:scalar_lower_bound_refresh}
\langle R(h_4) u, u \rangle_x
\ge
\bigl( \delta(D)\phi_1^{-2} - \pi^2/L^2 \bigr)\phi_2^{-4}
\,|u|_x^2 .
\end{equation}

The curvature term $\mathcal R_E$ is estimated following
\cite{llarull1998sharp}.
Fix $x \in \Gamma \times S^1_\rho$ and choose an $h_4$--orthonormal frame
$\{e_1,e_2,e_3,e_4\}$ near $x$ such that $(\nabla e_j)_x=0$,
with $\{e_1,e_2,e_3\}$ tangent to $\Gamma$ and $e_4$ tangent to $S^1_\rho$.
Likewise, choose an orthonormal frame
$\{\varepsilon_1,\varepsilon_2,\varepsilon_3,\varepsilon_4\}$ near $H(x)$
satisfying $(\nabla \varepsilon_j)_{H(x)}=0$ and
\[
H_* e_j = \mu_j \varepsilon_j
\qquad (1 \le j \le 4),
\]
for positive scalars $\mu_j$.

For $1 \le i \neq j \le 3$, the area--contraction of $G$ along $\Gamma$ implies
\begin{equation}\label{eq:mu_bound_Gamma}
\mu_i \mu_j \le l^2 \phi_1^{-2} \phi_2^{-4}.
\end{equation}
Moreover, since $\Gamma$ is compact, the map $G$ is $c$--contracting for some
constant $c$, and hence for $1 \le i \le 3$ we have
\begin{equation}\label{eq:mu_bound_circle}
\mu_i \mu_4 \le \frac{c}{\rho}.
\end{equation}

Combining \eqref{eq:mu_bound_Gamma} and \eqref{eq:mu_bound_circle} with the
estimates in~\cite{llarull1998sharp}, we obtain
\[
\langle \mathcal R_E u, u \rangle_x
\ge
-\frac14 \bigl(
l^2\phi_1^{-2}\phi_2^{-4} + 6c/\rho
\bigr)\,|u|_x^2 .
\]
Substituting this inequality and \eqref{eq:scalar_lower_bound_refresh} into
\eqref{eq:bochner_refresh}, we arrive at
\begin{equation}\label{eq:dirac_lower_bound_refresh}
\begin{aligned}
\bigl\langle \slashed D_E^{\,2} u , u \bigr\rangle
\;\ge\; \frac14 \Biggl(
&\Bigl\langle
\Bigl(
\delta(D)\phi_1^{-2}
- \frac{\pi^2}{L^2}
- 6l^2\phi_1^{-2}
\Bigr)
\phi_2^{-2} u ,
\; \phi_2^{-2} u
\Bigr\rangle
\\
&\qquad\qquad
- \frac{6c}{\rho}\,\|u\|^2
\Biggr).
\end{aligned}
\end{equation}

Since
\[
\ker \slashed D_{E^+}
\subset \ker \slashed D_E
\subset \ker \slashed D_E^2 ,
\]
the index computation implies $\operatorname{ind}(\slashed D_{E^+}) \neq 0$.
As $\phi_2>0$ and $\rho$ can be taken arbitrarily large,
inequality~\eqref{eq:dirac_lower_bound_refresh} forces the existence of a point
$x$ such that
\[
\delta(D)\phi_1^{-2}(x) - \frac{\pi^2}{L^2} - 6l^2\phi_1^{-2}(x) \le 0 .
\]
Letting $L \to \infty$ yields $\delta(D) - 6l^2 \le 0$, and sending $D \to \infty$
completes the proof in the product case.

\medskip
\noindent\textbf{Case~2.}
Finally, we reduce the general case to the product case.

Let
\[
f=(f_1,f_2): X \longrightarrow N \times \mathbb{R}
\]
be a proper smooth map of non-zero degree.
By properness, there exists $T \gg 1$ such that $f$ is transverse to
$\partial N \times \{T\}$ on the end $N \times [T,\infty)$.
Set
\[
X_T := f_2^{-1}\bigl([T,\infty)\bigr),
\qquad
\Sigma_T := f_2^{-1}(\{T\}).
\]
Then $X_T$ is complete, non-compact with boundary $\Sigma_T$, and the
positive scalar curvature lower bound is preserved.

We form the double $\widehat X_T$ of $X_T$ along $\Sigma_T$.
After smoothing the metric near the doubling locus, we obtain a complete
manifold with scalar curvature still bounded below by a positive constant.
The map $f$ extends to a proper smooth map
\[
\widehat f : \widehat X_T \longrightarrow \partial N \times \mathbb{R}
\]
of non-zero degree. Identifying $\partial N \times \mathbb{R}$ with a product $P \times \mathbb{R}$,
this reduces the situation to the product case treated in Case~1.
The conclusion then follows immediately.
This completes the proof in the general case.
\end{proof}
The proof of Theorem~\ref{thm:quantitative_sys_at_infinity} is entirely local
near the cylindrical end.
As a consequence, the same argument yields the following relative version.
\begin{cor}
\label{cor:quantitative_sys_relative}
Let $(M^5,\partial M)$ be a non-compact Riemannian manifold with compact boundary and
let
\[
f \colon (M,\partial M) \to (S^3\times S^1 \times \mathbb R_{\ge 0},\, S^3\times S^1)
\]
be a smooth non–zero degree map.
Assume that $M$ admits a complete metric $g$
with uniformly positive scalar curvature
$
R_g \ge \sigma >0 .
$

Then the $1$–dilation of $f$ satisfies
\[
\operatorname{Dil}_1(f) \;\ge\; \sqrt{\frac{\sigma}{6}} .
\]
\end{cor}
\begin{ex}\label{ex:quantitative_sys_end}
We explain how Theorem~\ref{thm:Ray_obstruction} in dimension \( n = 4 \)
follows from Corollary~\ref{cor:quantitative_sys_relative}.
Assume \( n = 4 \).
By hypothesis, the Poincar\'e dual at infinity of the proper ray \( \gamma \)
admits a factorization
\[
\alpha_\gamma = u^1 \cup u^2 \cup u^3,
\qquad u^i \in H^1_\infty(E).
\]
Each class \( u^i \) can be represented, via an Eilenberg--MacLane space,
by a map
\[
f_i \colon E \to S^1 .
\]

Let \( r \colon E \to \mathbb R_{\ge 0} \) be a proper function obtained by
smoothing the distance to a fixed compact separating hypersurface
near the beginning of the end.
Combining these maps, we obtain a proper map
\[
f = (f_1,f_2,f_3,r)
\colon E \to S^1 \times S^1 \times S^1 \times \mathbb R_{\ge 0},
\]
whose degree is nonzero.
Arguing as in Step~3 of the proof of
Theorem~\ref{thm:quantitative_sys_at_infinity},
we may further assume that \( f \) induces a surjection on fundamental groups.

By passing to a sufficiently large covering and considering the lifted map,
we obtain a map
\[
\tilde f \colon \widetilde E \to S^3 \times \mathbb R
\]
whose dilation constant can be made arbitrarily small.
Taking the product with an additional \( S^1 \) factor in both the domain
and the target, we obtain a proper map between \(5\)--manifolds
\[
\widetilde f \colon E \times S^1
\to S^3 \times S^1 \times \mathbb R_{\ge 0},
\]
whose dilation constant is again arbitrarily small
(the \( S^1 \)--component is defined by an appropriate scaling).

If \( E \) admitted a complete metric with uniformly positive scalar curvature,
then the product metric on \( E \times S^1 \) would also have uniformly positive
scalar curvature.
Moreover, the \(1\)--dilation of \( \widetilde f \) could be made arbitrarily small.

This contradicts Corollary~\ref{cor:quantitative_sys_relative},
which provides a quantitative lower bound on the \(1\)--dilation
in terms of the scalar curvature lower bound.
We conclude that no such metric can exist on \( E \).

\end{ex}

\section*{Acknowledgments}
The author is deeply grateful to his advisor, Tsuyoshi Kato, for his constant support and encouragement. 
The author also expresses his gratitude to Rafe Mazzeo for his interest, encouragement, and warm hospitality during his visit to Stanford University. 
The author also thanks Daiki Irikura for his valuable comments on Section~\ref{sec:prelim}, and Douglas Stryker, Otis Chodosh, and Filippo Gaia for their interest and helpful comments on an earlier draft. 
This work was supported by JST SPRING, Grant Number JPMJSP2110.

\bibliographystyle{amsalpha}
\bibliography{reference3.bib}

\end{document}